\documentclass[a4paper,11pt]{amsart}

\usepackage{amssymb,latexsym}
\usepackage{amsthm}
\usepackage{amsmath}
\usepackage{amsfonts}
\usepackage{graphicx}
\usepackage[all]{xy}
\usepackage{chngcntr}
\usepackage{fancyhdr}
\usepackage[top=1in,bottom=1in,left=1.2in,right=1.2in]{geometry}
\usepackage[usenames,dvipsnames,pdftex]{xcolor}
\usepackage{tikz,ifthen}
\usepackage{tikz-cd}
\usepackage{color}
\usepackage{pgfplots}
\usepackage{tikz}
\usetikzlibrary{matrix,arrows,decorations.pathmorphing}

\newtheorem{thm}{Theorem}[section]
\newtheorem{lem}[thm]{Lemma}
\newtheorem{definition}[thm]{Definition}
\newtheorem{example}[thm]{Example}
\newtheorem{prop}[thm]{Proposition}
\newtheorem{cor}[thm]{Corollary}
\newtheorem{rem}[thm]{Remark}

\def\Q{\mathbb{Q}}
\def\F{\mathbb{F}}
\def\R{\mathbb{R}}
\def\Z{\mathbb{Z}}
\def\A{\mathbb{A}}

\def\C{\mathbb{C}}
\def\G{\mathbb{G}}

\def\B{\mathcal{B}}
\def\Fl{\mathcal{F}\ell}

\def\cL{\mathcal{L}}
\def\M{\mathcal{M}}
\def\N{\mathcal{N}}

\def\Hom{\mathrm{Hom}}

\def\Lie{\mathrm{Lie}}

\def\GL{\mathrm{GL}}
\def\PGL{\mathrm{PGL}}
\def\SL{\mathrm{SL}}
\def\Perf{\mathrm{Perf}}

\def\Spa{\mathrm{Spa}}

\def\Spec{\mathrm{Spec}}

\def\Par{\mathrm{Par}}
\def\Bun{\mathrm{Bun}}
\def\Gr{\mathrm{Gr}}
\def\dR{\mathrm{dR}}
\def\Berk{\mathrm{Berk}}
\def\ad{\mathrm{ad}}
\def\adm{\mathrm{adm}}
\def\HN{\mathrm{HN}}
\def\Osc{\mathrm{Osc}}

\def\lan{\langle}
\def\ran{\rangle}

\def\lra{\longrightarrow}
\def\ra{\rightarrow}
\def\ov{\overline}

\def\wh{\widehat}
\def\wt{\widetilde}
\def\st{\stackrel}
\def\tr{\textrm}

\usepackage[breaklinks,colorlinks,linkcolor=blue,anchorcolor=blue,citecolor=blue,urlcolor=blue]{hyperref}

\begin{document}

\title{Bruhat-Tits buildings and $p$-adic period domains}
\author{Xu Shen}
\author{Ruishen Zhao}
\date{}
\address{Morningside Center of Mathematics\\
	Academy of Mathematics and Systems Science\\
	Chinese Academy of Sciences\\
	No. 55, Zhongguancun East Road\\
	Beijing 100190, China}
\address{University of Chinese Academy of Sciences, Beijing 100149}
\email{shen@math.ac.cn}

\address{Morningside Center of Mathematics\\
	Academy of Mathematics and Systems Science\\
	Chinese Academy of Sciences\\
	No. 55, Zhongguancun East Road\\
	Beijing 100190, China}
 
\email{zrs13@tsinghua.org.cn}

\renewcommand\thefootnote{}
\footnotetext{2020 Mathematics Subject Classification. Primary: 11E95; Secondary: 14G22.}

\renewcommand{\thefootnote}{\arabic{footnote}}

\begin{abstract}
Let $G$ be a connected reductive group over a $p$-adic local field $F$. R\'emy-Thuillier-Werner constructed embeddings of the (reduced) Bruhat-Tits building $\B(G,F)$ into the Berkovich spaces associated to suitable flag varieties of $G$, generalizing the work of Berkovich in split case. They defined compactifications of $\B(G,F)$ by taking closure inside these Berkovich flag varieties. We show that, in the setting of a basic local Shimura datum, the R\'emy-Thuillier-Werner embedding factors through the associated $p$-adic Hodge-Tate period domain. Moreover, we compare the boundaries of the Berkovich compactification of $\B(G,F)$  with non basic Newton strata. In the case of $\GL_n$ and the cocharacter $\mu=(1^d, 0^{n-d})$ for an integer $d$ which is coprime to $n$, we further construct a continuous retraction map from the $p$-adic period domain to the building. This reveals new information on these $p$-adic period domains, which share many similarities with the Drinfeld spaces.
\end{abstract}

\maketitle
\setcounter{tocdepth}{1}
\tableofcontents

\section{Introduction}

  Let $p$ be a fixed prime number. In this paper, we continue to study the geometry of $p$-adic period domains as initiated from \cite{cfs2021}.  
  More precisely, we compare $p$-adic period domains with the Bruhat-Tits building (\cite{bt1972, bt1984}) of the associated $p$-adic reductive group $G$.  Both are
  candidates of  the so called ``$p$-adic symmetric spaces'' for $G$, although the $p$-adic analogues of symmetric spaces that they show are in different aspects ($p$-adic Hodge theoretic and analytic vs. combinatorial and topological), cf. \cite{RZ96} and \cite{Yu, rtw2010}.  Previously, the only known relation between $p$-adic period domains and Bruhat-Tits buildings was in the Drinfeld case \cite{Dri74, Ber95}. Here we explore the link between these objects in the general setting. Our results shed new lights on both theories.
 \\
 
 Let $F$ be a finite extension of $\Q_p$, $G$ a connected reductive group over $F$, and $\{\mu\}$ the geometric conjugacy class of a minuscule cocharacter $\mu$ of $G$. Attached to the pair\footnote{Here we work with $\mu^{-1}$ temporarily, as later we will mainly work with an opposite parabolic, so we reserve $(G,\{\mu\})$ for later use.} $(G,\{\mu^{-1}\})$,  we have the  flag variety $\Fl(G,\mu^{-1})$ defined over the reflex field $E=E(G,\{\mu\})$, the field of definition of $\{\mu\}$.  Let $\breve{E}$ be the completion of maximal unramified extension of $E$.
 We consider the associated $p$-adic analytic space $\Fl(G,\mu^{-1})^{\tr{ad}}$ over $\breve{E}$, viewed as an adic space in the sense of Huber \cite{Hub96}. Let $b\in G(\breve{F})$ be an element up to $\sigma$-conjugacy (where $\breve{F}$ is defined similarly as $\breve{E}$, and $\sigma$ is the Frobenius for the extension $\breve{F}|F$), then Scholze constructed an open subspace \[\Fl(G, \mu^{-1}, b)^{\adm} \]
 of  $\Fl(G,\mu^{-1})^{\tr{ad}}$, cf. \cite{sw2020} Theorem 22.6.2 and page 220. See also \cite{cfs2021} section 3, where one describes $\Fl(G, \mu^{-1}, b)^{\adm}$ by modifications of $G$-bundles on the Fargues-Fontaine curve. 
 The space $\Fl(G, \mu^{-1}, b)^{\adm}$ is called the admissible locus, or the $p$-adic period domain attached to $(G, \{\mu^{-1}\}, b)$.
 Here to ensure $\Fl(G, \mu^{-1}, b)^{\adm}\neq \emptyset$, the data $b$ and $\{\mu^{-1}\}$ have to be compatible in the sense that the $\sigma$-conjugacy class $b\in B(G,\mu^{-1})$, cf. \cite{sw2020} Definition\footnote{Note that the sign convention there is different from us.} 24.1.1. Here $B(G,\mu^{-1})$ is the Kottwitz set as introduced in \cite{kott1997} section 6.
 For any open compact subgroup $K\subset G(F)$, there is a rigid analytic space $\M_K$ over $\breve{E}$, together with an \'etale morphism (called the de Rham period morphism) \[\pi_{\dR, K}: \M_K\lra \Fl(G, \mu^{-1}, b)^{\adm}\] which is surjective. By construction, the associated  diamond (in the sense of \cite{Sch17}) $\M_K^\Diamond$ admits an interpretation as a  moduli space of $p$-adic $G$-shtukas, cf. \cite{sw2020} Lectures 11 and 23. 
 If $G=D^\times$ with $D$ the division algebra over $F$ of invariant $\frac{1}{n}$, $\mu^{-1}=(1, \cdots, 1, 0)$, with the uniquely determined $b$, then the admissible locus $\Fl(G, \mu^{-1}, b)^{\adm}$ and the \'etale covers $\M_K$ were studied by Drinfeld in \cite{Dri76} as moduli spaces of $p$-divisible groups with certain additional structure. 
 
 Historically, $p$-adic period domains were studied intensively by Rapoport-Zink in \cite{RZ96} and Dat-Orlik-Rapoport in \cite{DOR10} (see also \cite{Rap, rap1997}), as vast generalizations of the work \cite{Dri76} of Drinfeld. Attached to the above $(G, \{\mu^{-1}\}, b)$, Rapoport and Zink constructed another open subspace
 \[ \Fl(G, \mu^{-1}, b)^{\tr{wadm}}\]
 of $\Fl(G,\mu^{-1})^{\tr{ad}}$, the weakly admissible locus, which parametrizes weakly filtered isocrystals (in the sense of Fontaine) with additional structure in this setting.  The complement of  $\Fl(G, \mu^{-1}, b)^{\tr{wadm}}$ in $\Fl(G,\mu^{-1})^{\tr{ad}}$ is a profinite union of certain Schubert varieties. By construction, we have \[\Fl(G, \mu^{-1}, b)^{\adm}\subset \Fl(G, \mu^{-1}, b)^{\tr{wadm}}. \] Moreover, the two have the same classical points, cf. \cite{cfs2021} Proposition 3.2.
 Assume that $b\in B(G, \mu^{-1})$ is \emph{basic} (cf. \cite{kott1985} section 5). In \cite{cfs2021} Chen, Fargues, and the first author here proved that $\Fl(G, \mu^{-1}, b)^{\adm}=\Fl(G, \mu^{-1}, b)^{\tr{wadm}}\quad \Leftrightarrow\quad  B(G, \mu^{-1})$ is fully Hodge-Newton decomposable.
 Here the right hand side is a group theoretic condition which roughly says that any non basic element $b'\in B(G, \mu^{-1})$ is decomposable with respect to $\{\mu^{-1}\}$. If this condition holds, the admissible locus $\Fl(G, \mu^{-1}, b)^{\adm}$ admits a simpler linear-algebra theoretic description since so does $\Fl(G, \mu^{-1}, b)^{\tr{wadm}}$. However, the classification of fully Hodge-Newton decomposable pairs $(G, \{\mu^{-1}\})$ shows that this condition is rather restrictive. For example, the pair $(G, \{\mu^{-1}\})$ with $G=\GL_5$ and $\mu^{-1}=(1,1,0,0,0)$ is not fully Hodge-Newton decomposable. In this case, we have \[\Fl(G, \mu^{-1}, b)^{\adm}\subsetneq\Fl(G, \mu^{-1}, b)^{\tr{wadm}},\] see \cite{hartl2013} Example 6.7. So far, beyond the fully Hodge-Newton decomposable case, the admissible locus  $\Fl(G, \mu^{-1}, b)^{\adm}$ is still very mysterious.
 
 Assume that $b$ is \emph{basic} as above. Let $G_b$ be the reductive group over $F$ defined as the $\sigma$-centralizer of $b$, which is then an inner form of $G$ as $b$ is basic. The group $G_b(F)$ acts naturally on $
 \Fl(G,\mu^{-1})^{\tr{ad}}$ by the inclusion $G_b(F)\subset G(\breve{E})$. This $G_b(F)$-action stabilizes $\Fl(G, \mu^{-1}, b)^{\adm}$. One of the main results of this paper roughly says that, we can find the Bruhat-Tits building of $G_b$ \[\B(G_b, F)\] ``inside'' the admissible locus $\Fl(G, \mu^{-1}, b)^{\adm}$.
 \\
 
Let us describe our results more precisely. We slightly change our perspective.  Consider the Bruhat-Tits building $\B(G, F)$ and the Berkovich space $\Fl(G, \mu)^\Berk$ (\cite{Ber90, Ber93}) attached to the flag variety $\Fl(G, \mu)$ over $E$. Here, it is natural to work with Berkovich flag variety $\Fl(G, \mu)^\Berk$,  as the underlying topological space is Hausdorff and compact, thus more suitable to be compared with the building $\B(G, F)$, which is  locally compact and Hausdorff.
In \cite{rtw2010} R\'emy-Thuillier-Werner constructed a continuous map\footnote{Note that we need the most subtle case of the construction in \cite{rtw2010}:  the flag variety $\Fl(G, \mu)$ is not necessary of rational type over $E$, and we need to work with the building over the subfield $F$ of $E$.} of topological spaces
\[\theta: \B(G, F)\lra  \Fl(G, \mu)^\Berk.\]
Roughly, one can view $\B(G, F)\subset \B(G, E)$ (see \cite{rtw2010} 1.3.4), and for any point $x\in \B(G, E)$, by \cite{rtw2010} Theorem 2.1 there is an affinoid subgroup \[G_x\subset G_E^\Berk\] which admits a unique Shilov boundary point $\wt{\theta}(x)$. Then one defines $\theta(x)\in \Fl(G, \mu)^\Berk$ as the ``projection''\footnote{More precisely, if $G$ is quasi-split over $E$, then one can take an $E$-rational point of $\Fl(G, \mu)$ to identify  $\Fl(G, \mu)=G_E/P_\mu$; the definition of $\theta$ does not depend on the choice of the  $E$-rational point, cf. \cite{rtw2010} Proposition 2.4. In the general case, one first works over an extension $E'|E$ which splits $G$, then uses some descent argument to back to $E$, cf. \cite{rtw2010} 2.4.4.} of $\wt{\theta}(x)$.  This construction generalizes the previous work of Berkovich \cite{Ber90}, who only treated split groups over $F$ and used a different approach. If the conjugacy class $\{\mu\}$ is non-degenerate, in the sense that it is non trivial on each $F$-quasi simple factors of $G$, then the map $\theta$ is an embedding of topological spaces, cf. \cite{rtw2010} Proposition 3.29.

On the other hand, in \cite{cs2017} Caraiani-Scholze defined the Newton stratification on the adic space $\Fl(G, \mu)^\ad$ over $E$ indexed by $B(G,\mu^{-1})$, which factors through $\Fl(G, \mu)^\Berk$. Thus we have a decomposition of $\Fl(G, \mu)^\Berk$ into locally closed subspaces \[\Fl(G, \mu)^\Berk=\coprod_{b\in B(G,\mu^{-1})}\Fl(G,\mu)^b,\] with the unique open stratum $\Fl(G,\mu)^{b_0}$ given by the basic element $b_0\in B(G,\mu^{-1})$, which we call the Hodge-Tate period domain for the local Shimura datum $(G, \{\mu^{-1}\}, b_0)$, since it serves as the target of the Hodge-Tate period morphism for the local Shimura variety at infinite level (cf. \cite{shen2023}). Here one can again recover the case of Drinfeld space: by considering $G=\GL_n$ and $\mu=(1,0,\cdots, 0)$, the space $\Fl(G,\mu)^{b_0}$ is the $n-1$ dimensional Drinfeld space over $F$.

Now,  the natural question is that which Newton strata $\Fl(G,\mu)^b$ contain the image $\theta(\B(G,F))$ of the Bruhat-Tits building under the R\'emy-Thuillier-Werner map? It turns out this has a pleasant answer.
\begin{thm}[Theorem \ref{thm BT vs p-adic period}]\label{thm 1.1}
The continuous map $\theta$ factors through the open Newton stratum $\Fl(G,\mu)^{b_0}$, i.e. $\theta(\B(G,F))\subset \Fl(G,\mu)^{b_0}$.
\end{thm}
One key point in the proof is the base change functoriality of the map $\theta$: for any non-archimedean extension $K|E$, the following diagram
\[\xymatrix{
	\B(G,K)\ar^{\theta_K}[r] & \Fl(G,\mu)^\Berk_K\ar[d]^{\tr{pr}_{K|E} } \\
	\B(G,F)\ar[r]^\theta \ar@{^{(}->}[u] & \Fl(G, \mu)^\Berk
}\]
commutes. This is also the key property of the map $\theta$ in the construction in \cite{rtw2010}. For any point $x\in \B(G,F)$, let $x_K\in \B(G,K)$ be its image under the embedding $\B(G,F)\hookrightarrow \B(G,K)$, and denote $y=\theta_K(x_K)$. The next key point is that one can compute $y$ explicitly for large enough $K$. Let $\mathcal{H}(y)$ be the complete residue field at $y$, and $\wt{\mathcal{H}(y)}$ the residue field in characteristic $p$ of $\mathcal{H}(y)$. By the explicit computation we get \[\tr{tr.deg}(\wt{\mathcal{H}(y)}|\wt{K})=\dim\,\Fl(G, \mu),\] where $\wt{K}$ is the residue field of $K$. Then we work in adic spaces and consider the quotient map $\Fl(G, \mu)_K^\ad\ra \Fl(G, \mu)^\Berk_K$. We view $y$ as a point of $\Fl(G, \mu)_K^\ad$ (by the natural discontinuous inclusion from the Berkovich space into adic space) and let $\ov{\{y\}}$ be its closure. By some arguments on dimensions of locally spectral spaces (for example see \cite{ext2022}), we get \[\dim\,\ov{\{y\}}=\tr{tr.deg}(\wt{\mathcal{H}(y)}|\wt{K}),\] which is thus $\dim\,\Fl(G, \mu)$. By construction, $\ov{\{y\}}$ entirely lie in some Newton stratum, which has to be the open stratum $\Fl(G, \mu)^{b_0}_K$ by the dimension formula of Newton strata, cf. Proposition \ref{prop Newton}. Roughly speaking, for any $x\in \B(G,F)$, the associated analytic point $\theta(x)\in \Fl(G, \mu)^\Berk_K$ is ``very generic'' in some sense, so that it can not lie in non basic Newton strata. Thus, one sees that Theorem \ref{thm 1.1} holds essentially due to a topological reason.

Even better, we can compare the boundaries of $\theta(\B(G,F))$ and $\Fl(G,\mu)^{b_0}$. Let $t$ be the type\footnote{See \cite{rtw2010} 1.1.3, which means a connected component of the total flag variety $\Par(G)$} of the flag variety $\Fl(G,\mu)$ and denote $\B_t(G,F)=\theta(\B(G,F))$. Consider the closure $\ov{\B_t(G,F)}$ of $\B_t(G,F)$ inside the compact Hausdorff space $\Fl(G, \mu)^\Berk$. If $\{\mu\}$ is non-degenerate, the space $\ov{\B_t(G,F)}$ defines a compactification of the building $\B(G,F)$. By \cite{rtw2010} Theorem 4.1, we have a stratification 
\[\overline{\B_t(G,F)}= \coprod_{\substack{Q\in \Par(G)(F)\\ Q\; \tau\tr{-relevant}}} \B_\tau(Q_{ss},F),\]
where $\tau$ is an $F$-rational type uniquely determined by $\{\mu\}$ (or $t$),  $Q$ runs over the set of $\tau$-relevant $F$-parabolic subgroups of $G$, and $Q_{ss}$ is the Levi quotient of $Q$. For the notion of $\tau$-relevant parabolic subgroups, see \cite{rtw2010} subsection 3.2 or our subsection 2.2.
For $Q=G$,  $\B_\tau(Q_{ss},F)=\B_t(G, F)$. 
The natural $G(F)$-action on $\B(G,F)$ extends uniquely to an action on $\overline{\B_t(G,F)}$. For a proper $Q\subsetneq G$, one may ask which non basic Newton strata contain the boundary stratum $\B_\tau(Q_{ss},F)$?
\begin{thm}[Theorem \ref{thm boundaries}]\label{thm 1.2}
With the above notations, we have 
\[ \B_\tau(Q_{ss},F)\subset \Fl(G,\mu)^b,\]
where let $M_Q$ be the Levi subgroup of $Q$, then $b\in B(G,\mu^{-1})$ is the image of the basic element $b_{M_Q,0}\in B(M_Q,\mu^{-1})$ under the map $B(M_Q,\mu^{-1})\ra B(G,\mu^{-1})$ induced by the inclusion $M_Q\subset G$.  

Moreover, one can characterize all $b\in B(G,\mu^{-1})$ such that the corresponding Newton stratum $\Fl(G,\mu)^b$ contains a boundary stratum of $\B_t(G, F)$. 
\end{thm}
The proof of this theorem is based on Theorem \ref{thm 1.1}, and the functoriality of the Newton stratification, cf. Lemma \ref{lem newton sets} and Proposition \ref{natural}. Here, the non basic elements $b\in B(G,\mu^{-1})$ such that the corresponding Newton stratum $\Fl(G,\mu)^b$ contains a boundary stratum of $\B_t(G, F)$ are characterized in Definition \ref{def SR elements}, which we call strongly regular elements. It turns out that for quasi-split $G$, strongly regular elements are Hodge-Newton decomposable with respect to $\{\mu^{-1}\}$. In general, a strongly regular element coming from a proper Levi is non basic, cf. Lemma \ref{lem SR non basic}.  As a corollary, the subset $\B_t(G,F)$ is  closed in $\Fl(G, \mu)^{b_0}$. If $G$ is quasi-split, it is also closed in the larger semistable locus\footnote{This is the analogue of the previous weakly admissible locus $\Fl(G, \mu^{-1}, b_0)^{\tr{wadm}}$, for more information about the precise relation, see \cite{shen2023}.} $\Fl(G, \mu)^{\HN=b_0}$, which is the maximal (open) stratum of the Harder-Narasimhan stratification on $\Fl(G, \mu)^\Berk$. For more information of the  Harder-Narasimhan stratification, see \cite{DOR10} and \cite{shen2023}. Here to prove\footnote{In fact, by a different approach it may be possible to prove the stronger statement without the quasi-split assumption, cf. Remark \ref{rem stronger closed}.} the second stronger statement, we apply Theorem 1.3 of \cite{eva2024}.

Consider the case $G=\GL_n$ and $\mu=(1^d,0^{n-d})$ with $(d,n)=1$.  In this case, 
we can construct a retraction map for the embedding $\theta: \B(G,F)\hookrightarrow \Fl(G,\mu)^{b_0}$.
\begin{thm}[Theorem \ref{thm retract continuous}, Theorem \ref{retract}, Proposition \ref{prop retract Drinfeld}]\label{thm 1.3}
Under the above assumptions, there exists a continuous map
\[r:  \Fl(G,\mu)^{b_0}\lra \B(G,F)\] such that $r\circ \theta=\tr{Id}$. Moreover, for $d=1$, this map $r$ coincides with the Drinfeld retraction map as constructed in \cite{Dri74} using an explicit formula by restriction of norms (see also \cite{Ber95} and \cite{DOR10} XI.3).
\end{thm}
The construction of the map $r$ is inspired from the works of van der Put-Voskuil \cite{pv1992} and Voskuil \cite{vos2000}. First of all, we have the inclusion \[\Fl(G,\mu)^{b_0}\subset \Fl(G,\mu)^{ss}:=\Fl(G,\mu)^{\HN=b_0},\] cf. \cite{shen2023} Proposition 3.4. On the other hand, we have the subspace of stable locus $\Fl(G,\mu)^{s}\subset \Fl(G,\mu)^{ss}$.
Then under the assumption $(d, n)=1$, the semistable locus coincides with the stable locus \[\Fl(G,\mu)^{s}=\Fl(G,\mu)^{ss}.\] Moreover, both admit descriptions in terms on geometric invariant theory; for these facts see \cite{pv1992} sections 1 and 2  (see also \cite{totaro} and \cite{DOR10} Theorem 9.7.3 for the statement on GIT).  Then it suffices to construct a retraction map $r: \Fl(G,\mu)^{s}\ra \B(G,F)$ with the desired properties. To this end,
 we proceed as \cite{vos2000} by firstly constructing continuous maps $r_A$ for each department $A\subset \B(G, F)$, cf. Propositions \ref{prop continuous retract torus} and \ref{prop properties retract torus}. To construct these continuous $r_A$, we use crucially the fact that the semistable locus coincides with the stable locus. Moreover, we
  in fact need to use some GIT arguments over the integral base $O_F$. Next, we
show these maps $r_A$ are compatible, cf. Lemma \ref{lemma compatibility}, which is very crucial to the following. Based on these careful analysis on Bruhat-Tits buildings and GIT, we can construct analytic subspaces $Y_z\subset \Fl(G,\mu)^{s}$ for each $z\in \B(G, F)$, such that $\Fl(G,\mu)^{s}=\coprod_{z\in \B(G, F)}Y_z$ as in \cite{vos2000}. Then we define a map \[r: \Fl(G,\mu)^{s}\lra \B(G, F)\] by contracting these $Y_z$, and prove that it is continuous, cf. Theorem \ref{thm retract continuous}. By restricting to the open subspace $\Fl(G, \mu)^{b_0}$, we get a continuous map $r$. Then we check that this indeed gives a retraction map for $\theta$, cf. Theorem \ref{retract}, and it coincides with the Drinfeld map in the case $d=1$, cf. Proposition \ref{prop retract Drinfeld}. Note that in the Drinfeld case, the definition of $r$ is quite simple, as it is given by an explicit restriction formula of norms, cf. \cite{Dri74, Ber95}. 

 In \cite{vos2000} under the assumption $\Fl(G,\mu)^{ss}=\Fl(G,\mu)^{s}$, a map from the rigid analytic version of $\Fl(G,\mu)^{s}$ to the building $\B(G,F)$ was announced. Here we work with Berkovich spaces, in contrast with \cite{vos2000}. Moreover, many crucial properties and arguments such as in Propositions \ref{prop continuous retract torus} and \ref{prop properties retract torus} and Lemma \ref{lemma compatibility} are missing in \cite{vos2000}. 
Furthermore, in the setting of \cite{pv1992}  and \cite{vos2000},  this map could not be called a retraction map, since at that time there was no embedding of $\B(G,F)$  into $\Fl(G,\mu)^{s}$ at all.

By Theorems \ref{thm 1.1} and \ref{thm 1.3}, when $\{\mu\}$ varies, the same building $\B(\GL_n, F)$ admits embeddings and receives retractions to/from different $p$-adic period domains $\Fl(\GL_n, \mu)^{b_0}$. This reveals that these $p$-adic period domains admit similar combinatorial structure in some sense.   To study $\Fl(\GL_n, \mu)^{b_0}$, we are reduced to study the fibers of the retraction map, which are \[\Fl(\GL_n, \mu)^{b_0}\cap Y_z\] for $z\in \B(\GL_n, F)$.  In the extreme\footnote{The flag variety $=\mathbb{P}^{n-1}$ has the minimal dimension, which equals to the dimension of $\B(\GL_n, F)$.} case $d=1$, we have $\Fl(\GL_n, \mu)^{b_0}=\Fl(\GL_n,\mu)^{s}$ and these fibers are clear; in fact this was one way to construct the $p$-adic analytic structure of the Drinfeld space, cf. \cite{Dri74, ss91}.
On the other hand, the retraction map $r$ may be useful to offer a new\footnote{The other ``standard'' approach is to use the combinatorial structure of the complement of the semistable locus, cf. \cite{DOR10} Chapter X. } approach for computing the $\ell$-adic cohomology of the semistable locus $\Fl(\GL_n, \mu)^{ss}$, as what Dat did in \cite{dat2006} in the Drinfeld case $d=1$.
\\

So far we mainly state our results in the Hodge-Tate setting. In section \ref{section de Rham}, we will briefly translate the previous constructions and results to the de Rham setting, discussing the relations between the building $\B(G_{b_0},F)$ and the admissible locus $\Fl(G,\mu^{-1}, b_0)^\adm$ for a basic local Shimura datum $(G, \{\mu^{-1}\}, b_0)$  as in the very beginning of this introduction. In particular, in the setting of Theorem \ref{thm 1.3}, the whole picture is quite similar to the Drinfeld case as in \cite{Dri76}, although the retraction map is much more complicated in general, and it is not any more a fully Hodge-Newton decomposable case.
\\

We briefly describe the structure of this paper. In section \ref{ber}, we review the Berkovich map $\theta$ and Berkovich compactification of the building $\B(G,F)$, following the paper of R\'emy-Thuillier-Werner \cite{rtw2010}. In section \ref{newtonstrata}, we introduce the Newton stratification on the $p$-adic flag variety $\Fl(G,\mu)^\Berk$, and prove Theorem \ref{thm 1.1}. In section \ref{boundary}, we compare the boundary strata of the Berkovich compactification and non basic Newton strata, proving the Theorem \ref{thm 1.2}. In section \ref{retraction}, we consider some special examples, with $G=\GL_n$ and $\mu=(1^d, 0^{n-d})$ with $(d, n)=1$. We construct a continuous retraction map from $\Fl(G,\mu)^{b_0}$ to $\B(G, F)$ inspired by the works \cite{pv1992} and \cite{vos2000}, and prove Theorem \ref{thm 1.3}. We also make some remarks about potential applications of such a retraction map. Finally in  section \ref{section de Rham}, we briefly translate the previous results to the de Rham setting, and summarize the relations between the building $\B(G_{b_0},F)$ and the admissible locus $\Fl(G,\mu^{-1}, b_0)^\adm$ for a basic local Shimura datum $(G, \{\mu^{-1}\}, b_0)$. We also discuss some related open problems on  cohomological applications in the setting of Fargues-Scholze \cite{FS}.\\
 \\
\textbf{Acknowledgments.} We would like to thank the institutes YMSC and BICMR in Beijing, as this work was started after both authors attended some talks in these institutes during January 2024.  We would like to thank Laurent Fargues and Peter Scholze for helpful conversations during the preparation of this work.    The first author was partially supported by the National Key R$\&$D Program of China 2020YFA0712600, the CAS Project for Young Scientists in Basic Research, Grant No. YSBR-033, and the NSFC grant No. 12288201.\\
 \\
\textbf{Notations}:
We will use the following notions (mainly follow the conventions in \cite{cfs2021} and \cite{rtw2010}):
\begin{itemize}
	\item 

$F$ is a finite extension of $\mathbb{Q}_p$ with residue field $\mathbb{F}_q$.

\item  $\overline{F}$ is an algebraic closure of $F$ with Galois group $\Gamma=\mathrm{Gal}(\overline{F}|F)$.

\item $\breve{F}=\widehat{F^{ur}}$ is the completion of the maximal unramified extension $F^{ur}$ of $F$.

\item  $G$ is a connected reductive group over $F$.

\item $H$ is a quasi-split inner form of $G$, equipped with an inner twisting $G_{\ov{F}}\st{\sim}{\ra} H_{\ov{F}}$.

\item $A\subset T\subset B$ are maximal $F$-split torus, maximal torus, and Borel subgroup of $H$. Sometimes we also use $A$ and $T$ to denote the corresponding notions for $G$.

\item $(X^\ast(T), \Phi, X_\ast(T), \Phi^\vee)$ is the absolute root datum of $H$, with positive roots $\Phi^+$ and simple roots $\Delta$ with respect to the choice of $B$; while $(X^\ast(A), \Phi_0, X_\ast(A), \Phi_0^\vee)$ is the relative root datum of $H$, with positive roots $\Phi^+_0$ and simple roots $\Delta_0$ with respect to the choice of $B$.

\item For a standard Levi subgroup $M$ of $H$, we note by a subscript $M$ the corresponding roots or coroots appearing in $\Lie\,M$.

\item Fix a minimal $F$-parabolic and $F$-Levi subgroups $M_0\subset P_0 \subset G$. Standard $F$-parabolic and $F$-Levi subgroups of $G$ are those containing $P_0$ and $M_0$ respectively.

\item  $G_{ad}$ is the adjoint group associated to $G$.

\item  $\B(G,F)$ denotes the reduced Bruhat-Tits building for $G(F)$.  We have a homeomorphism \[\B(G,F)\cong \B(G_{ad},F)\] induced by the natural map $G\rightarrow G_{ad}$.

\item For a cocharacter $\mu$, we will use $\{\mu\}$ to denote its conjugacy class.


\item Let $\mathbb{D}$ be the pro-torus with character group $X^\ast(\mathbb{D})=\Q$. The Newton chamber of $G$ is defined by \[\N(G)=[\Hom(\mathbb{D}_{\ov{F}}, G_{\ov{F}})/G(\ov{F})\tr{-conjugacy}]^\Gamma.\]
Via the inner twisting between $G$ and $H$, we have an identification
\[\N(G)=\N(H)=X_\ast(A)^+_\Q.\]
This is equipped with the partial order $v_1\leq v_2$ if and only if $v_2-v_1\in \Q_{\geq 0}\Phi_0^+$.
\item $\pi_1(G)=X_\ast(T)/\lan \Phi^\vee\ran$ is the algebraic fundamental group of $G$, and $\pi_1(G)_\Gamma$ is its Galois coinvariant. Via the inner twisting between $G$ and $H$, we have identifications
\[\pi_1(G)=\pi_1(H), \quad \pi_1(G)_\Gamma=\pi_1(H)_\Gamma.\]

\item For an algebraic variety $X$ over a non-archimedean field $k$, $X^\Berk$ is the associated Berkovich space (\cite{Ber90}), and $X^\ad$ is the associated adic space (\cite{Hub96}).

\end{itemize}

\section{Bruhat-Tits buildings and $p$-adic flag varieties}
\label{ber}

In this section, we review the continuous map from Bruhat-Tits buildings to Berkovich flag varieties and the Berkovich compactification of  buildings constructed in \cite{rtw2010}. In loc. cit. the authors use
the language of \emph{types} for parabolic subgroups to define connected flag varieties. We will compare it with the language of cocharacters, which is more suitable in the setting of $p$-adic Hodge theory. We assume that the reader has some familiarities with the theory of Berkovich spaces, see \cite{Ber90, Ber93} (or \cite{rtw2010} subsection 1.2 for a very brief review). For basics on Bruhat-Tits buildings, we refer to \cite{tits1979,Yu, bt1972, bt1984} (or  \cite{rtw2010} subsection 1.3 for a very brief summary). See also \cite{rtw2015} for a survey of both theories.

\subsection{Embeddings of Bruhat-Tits buildings into Berkovich flag varieties}
Let  $G$ be a connected reductive group over $F$ as above. As in \cite{rtw2010} section 1.1, let \[\Par(G)\] denote the moduli space of all parabolic subgroups of $G$. This is a projective and smooth scheme over $F$. For any $F$-scheme $S$, $\Par(G)(S)$ consists of the set of smooth subgroups $P$ of $G_S=G\times_FS$ such that for any geometric point $\ov{s}$ of $S$, the quotient $G_{\ov{s}}/P_{\ov{s}}$ is a proper $\ov{s}$-scheme (in other words, $P_{\ov{s}}$ is a parabolic subgroup of $G_{\ov{s}}$). For a parabolic subgroup $P \subset G_{\overline{F}}$,  the \emph{type} of $P$ is the connected component (over $F$) of $\Par(G)$ containing $P$. Sometimes we say also $F$-types of parabolic subgroups of $G$, as this notion depends on the base field. The set of types of parabolic subgroups of $G$ is by definition the set of connected components of $\Par(G)$ over $F$, which is in natural bijection with the $\mathrm{Gal}(\ov{F}/F)$-stable subsets of the simple roots of $G_{\ov{F}}$.
The connected component corresponding to a given type $t$ is denoted by \[\Par_t(G).\] The type $t$ is called $F$-\emph{rational} if the flag variety $\Par_t(G)$ contains an $F$-point, which corresponds to a parabolic subgroup of $G$ over $F$ equivalently. 

In this paper, we will mainly produce flag varieties from cocharacters. Given a cocharacter $\mu: \G_m\ra G_{\overline{F}}$ for $G_{\overline{F}}$, we can associate it a parabolic subgroup of  $G_{\overline{F}}$ (see \cite{cs2017} section 2.1):
\[P_{\mu}=\{g| \lim_{t \to 0}ad(\mu(t))g \ \text{exists}\}.\]
It has a Levi component being the centralizer of $\mu$: \[M_{\mu}=\mathrm{Cent}_{G_{\overline{F}}}(\mu).\]
Each parabolic subgroup of $G_{\bar{F}}$ can be obtained in this way, but two different cocharacters may correspond to a common parabolic subgroup. For example, if we rescale $\mu$ by $\mu_1=\mu^{d}$ ($d$ is a positive integer), then $P_{\mu}=P_{\mu_1}$.
Passing to conjugacy  class, each conjugacy class $\{\mu\}$ will determine a conjugacy class of $P_{\mu}$, then a flag variety \[\Fl(G,\mu) \cong G_{\bar{F}}/P_{\mu}.\] It is a connected component of $\Par(G)_{\ov{F}}$, thus  projective and smooth over $\ov{F}$. Moreover, it is defined over the (local) reflex field $E=E(G,\{\mu\})$,  which is by definition the field of definition for the conjugacy class $\{\mu\}$ (thus a finite extension of $F$). In the following, we still denote by $\Fl(G,\mu)$ the corresponding flag variety over $E$.  If $G$ is quasi-split over $F$, then the flag variety $\Fl(G,\mu)$ contains an $E$-point by a result of Kottwitz (see \cite{DOR10} Lemma 6.1.5). For any algebraically closed field $L|E$, the set $\Fl(G,\mu)(L)$ can be described as the set of $G$-filtrations over $L$ of type $\{\mu\}$, cf. \cite{DOR10} Theorem 6.1.4.

Now we make  a comparison of notions. Each conjugacy class $\{\mu\}$ determines a flag variety $\Fl(G,\mu)$, which  correspond to a connected component of $\Par(G)_{\bar{F}}$ which is defined over $E$, thus an $E$-type $t_E$ for $G_E$ or $\Par(G_E) (=\Par(G)_E)$. This $E$-type may or may not be rational (depending on whether $\Fl(G,\mu)$ has an $E$-rational point).
We will work with the geometrically connected flag variety $\Fl(G,\mu)$ over $E$. It determines a unique $F$-type $t=t_\mu$ (which may not be $F$-rational) by considering all the Galois conjugates of $\mu$ (or of $t_E$). We have a natural morphism
\[\Fl(G,\mu)\hookrightarrow \Par_{t}(G)_E, \] which realizes $\Fl(G,\mu)$ as a connected component of  $\Par_{t}(G)_E$.
On the other hand, we have a projection map (for algebraic varieties over different base fields)
\[\tr{pr}_{E|F}:  \Fl(G,\mu)\ra \Par_{t}(G).\]
If $E=F$, then we have $t=t_E$ and $\Fl(G,\mu)=\Par_{t}(G)$.

Since our base field $F$ is a $p$-adic local field (locally compact), the \textbf{functoriality} assumption for Bruhat-Tits buildings in \cite{rtw2010} (section 1.3.4) is satisfied:
The Bruhat-Tits building construction forms
 a functor $\B(G,\cdot )$ from the category $\mathrm{Extfd}(F)$ of non-archimedean extensions of $F$ to the category of topological spaces, mapping a field extension $F_2|F_1$ to a $G(F_1)$-equivariant continuous \emph{injection} \[i_{F_1,F_2}: \B(G,F_1)\hookrightarrow \B(G,F_2).\]
This functorial property is frequently used in \cite{rtw2010} and will also be a useful technique for this paper. It enables us to convert the general case into split case and work with special vertices.

In section 2.2 of \cite{rtw2010}, the authors construct a canonical map \[\theta: \B(G,F) \longrightarrow G^{\Berk},\] which is continuous and $G(F)$-equivariant (for the conjugation action of $G(F)$ on $G^{\Berk}$). The construction proceeds in two steps: for any $x\in \B(G,F)$,
\begin{enumerate}
	\item[Step 1:] One constructs an affinoid subgroup $G_x\subset G^{\Berk}$ such that for any non-archimedean extension $L|F$, we have \[G_x(L)=\mathrm{Stab}_{G(L)}(x_L),\] where $x_L\in \B(G,L)$ is the image of $x$ under the canonical map $\B(G,F)\hookrightarrow \B(G,L)$; see \cite{rtw2010} Theorem 2.1.
	\item[Step 2:] The affinoid subgroup $G_x$ has a unique Shilov boundary \[\theta(x)\in G_x\subset G^{\Berk}.\] By \cite{rtw2010}  Proposition 2.4, this defines a continuous map $\theta: \B(G,F) \longrightarrow G^{\Berk}$.
\end{enumerate}
By \cite{rtw2010} Proposition 2.7 this map is functorial with respect to fields extensions: for any non-archimedean  extension $L|F$, the natural diagram 
\[\xymatrix{
\B(G,L)\ar[r]^{\theta_L}& G_L^{\Berk}\ar[d]^{\tr{pr}_{L|F}}\\
\B(G,F)\ar[u]^{i_{F,L}}\ar[r]^\theta& G^{\Berk}
}\]
is commutative.

Recall our conjugacy class of cocharacters $\{\mu\}$ and the associated flag variety $\Fl(G,\mu)$ over $E$. If this flag variety has an $E$-rational point $P$, 
then through the map \[\lambda_P: G_E\longrightarrow G_E/P\cong \Fl(G,\mu),\] after passing to the associated morphism of Berkovich spaces and by composing with the inclusion $i_{F,E}: \B(G,F)\hookrightarrow \B(G,E)$, we get a $G(F)$-equivariant continuous map:
\[\theta_{\mu,F,E}=\lambda_P\circ\theta_E\circ i_{F,E}: \B(G,F)\longrightarrow \Fl(G,\mu)^{\Berk}.\] It does not depend on the choice of the $E$-rational point $P$. In the general case, take a finite extension $L|E$ such that $\Fl(G,\mu)$ has an $L$-rational point. We have $\theta_{\mu, F, L}: \B(G,F)\longrightarrow \Fl(G,\mu)^{\Berk}_{L}$ as above. Then define \[\theta_{\mu, F, E}=\tr{pr}_{L|E}\circ \theta_{\mu, F, L}: \B(G,F)\longrightarrow \Fl(G,\mu)^{\Berk}.\] As in \cite{rtw2010} 2.4.4 this does not depend on the choice of $L$. We call $\theta_{\mu, F, E}$ \textbf{the Berkovich map}. It is canonical, only relies on the conjugacy class $\{\mu\}$. It has the following base change functoriality.

\begin{prop}
\label{bc1}

For any non-archimedean extension $L|E$, 
the following diagram is commutative:

\[\xymatrix{
\B(G,L) \ar[r]^{\theta_{\mu,L,L}} &\Fl(G,\mu)^{\Berk}_L\ar[d]^{\tr{pr}_{L|E}}\\
\B(G,F) \ar[u]^{i_{F,L}}\ar[r]^{\theta_{\mu,F,E}} &\Fl(G,\mu)^{\Berk}},\]
where $\theta_{\mu,L,L}$ is defined similarly as above.

\end{prop}
This can be proved exactly as \cite{rtw2010} Proposition 2.16. In fact, as in \cite{rtw2010} 2.4.2 and 2.4.4, we can consider the Berkovich map over the same base field $F$ (here $t=t_\mu$ as above): \[\theta_{t, F}: \B(G, F)\lra \Par_{t}(G)^\Berk.\] Then there is a similar commutative diagram
\[\xymatrix{
\B(G, E)\ar[r]^{\theta_{\mu,E,E}} &\Fl(G,\mu)^{\Berk} \ar[d]^{\tr{pr}_{E|F}}\\
\B(G,F)\ar[u]^{i_{F,E}}\ar[r]^{\theta_{t,F}} & \Par_{t}(G)^\Berk.
} \]
By construction, we have $\theta_{\mu, F,E}=\theta_{\mu, E, E}\circ i_{F,E}$ and thus $\theta_{t,F}=\tr{pr}_{E|F}\circ \theta_{\mu, F,E}$. The image $\theta_{\mu,F,E}(\B(G,F))$ is isomorphic to the image $\theta_{t,F}(\B(G,F))$ under the projection map $\tr{pr}_{E|F}$.
Indeed, this can be easily deduced from the base change diagram by choosing a common Galois extension $L$ of $E$ and $F$.
On the other hand,
for any non-archimedean field extension $L|E$,  as the above we get a map \[\theta_{\mu,F,L}:\B(G,F)\longrightarrow \Fl(G,\mu)^{\Berk}_L,\] which can be seen as a lift for the map $\theta_{\mu,F,E}$ respect to the projection map $\tr{pr}_{L|E}$. This already shows that the image of the Bruhat-Tits building is special. For a general subspace of $\Fl(G,\mu)^{\Berk}$, we do not have  such a canonical lift. 

Now, it is clear that by base change we can pass to split reductive groups to study the Berkovich map. In the split setting, we can make the map $\theta_{\mu,F, E}$ more explicitly.
Assume that $G$ is split over $F$ (so that $E=F$). 
We only need to consider $F$-cocharacters. Let $\mu$ be such a cocharacter,  $P_{\mu}$ the corresponding parabolic subgroup over $F$. Take a maximal split torus $T$ inside $P=P_\mu$ and it is also a maximal split torus for $G$. It will determine an apartment $A(T,F)$ inside the Bruhat-Tits building $\B(G,F)$. On the other hand, let $P^{op}$ denote the parabolic subgroup of $G$ opposite to $P$ with respect to $T$ and $N^{op}$ denote its unipotent radical. Let $t$ be the type of $P$. The $F$-morphism \[N^{op}\longrightarrow \Par_t(G) \cong \Fl(G,\mu), \quad  g\mapsto g P_{\mu} g^{-1}\] is an isomorphism onto an open subvariety (open Bruhat cell) of $\Fl(G,\mu)$, which we denote by $\Omega(T,P)$.

Let $\Phi (G,T)$ be the set of roots of $G$ with respect to $T$ and pick a special vertex $o$ in $\B(G,F)$ compatible with $T$ ($o$ lies in the apartment $A(T,F)$). Let $\Psi=\Phi(N^{op},T)$, the subset of $\Phi(G,T)$ determined by $N^{op}$. The special vertex $o$ will give us an integral model for $G$ and the choice of  an integral Chevalley basis in $\Lie\,G$ gives isomorphisms  \[\Omega(T,P)\cong N^{op}\cong\Spec\, F[(X_{\alpha})_{\alpha \in \Psi}].\]

\begin{prop}
\label{explicit}
Under the above assumptions, we have the following statements.
\begin{enumerate}
	\item 

The Berkovich map $\theta_{t,F}$ sends the point $o$ to the point of $\Omega(T,P)^{\Berk}$ corresponding to the multiplicative norm (Gauss point) \[F[(X_{\alpha})_{\alpha \in \Psi}]\rightarrow \mathbb{R}_{\geq 0},\quad   \sum_{v \in \mathbb{N}^{\Psi}}a_v X^{v}\mapsto \max_{v} |a_v|. \]

\item Use the point $o$ as the origin to  identify the apartment $A(T,F)$ with the vector space $V(T)=\mathrm{Hom}(X^{*}(T),\mathbb{R})$, the map \[V(T)\longrightarrow \Fl(G,\mu)^{\Berk}\] induced by $\theta_{t,F}$ sends an element $u$ of $V(T)$ to the point of $\Omega(P,T)^{\Berk}$ corresponding to the multiplicative norm \[F[(X_{\alpha})_{\alpha \in \Psi}]\rightarrow \mathbb{R}_{\geq 0},\quad   \sum_{v \in \mathbb{N}^{\Psi}}a_v X^{v}\mapsto \max_{v} |a_v|\prod_{\alpha \in \Psi} e^{v(\alpha)\langle u,\alpha\rangle}. \]
\end{enumerate}

\end{prop}

We refer to \cite{rtw2010} Proposition 2.17 for the proof. 

\begin{rem}
\begin{enumerate}
	\item 
This proposition shows that the Berkovich map will map an apartment into the corresponding open Bruhat cell. Moreover, each point lying in the image is ``\textbf{very generic}'', looks like a generalized Gauss point. This intuition later will help us to prove Theorems \ref{main1} and \ref{retract}.

\item Since $\B(G,F)$ is a union of apartments, the second statement already determines the image of  the Berkovich map. Even better, through the base change map, each point in $\B(G,F)$ will become a special vertex inside $\B(G,L)$ after a suitable non-archimedean field extension $L|F$ (cf. \cite{rtw2010} Proposition 1.6), so the first statement already suffices for many applications.
\end{enumerate}
\end{rem}

Before going on, we discuss when the Berkovich map $\theta_{t, F}: \B(G,F)\lra \Par_t(G)^\Berk$ is an embedding, slightly generalizing the discussions in \cite{rtw2010} 3.4.1 and 3.4.2.
If $G$ is a semisimple  $F$-group (in practice, we can pass to $G_{ad}$, which will not influence the Bruhat-Tits building or the Newton stratification on the $p$-adic flag variety to be introduced later), then there exists a unique finite family $(G_i)_{i\in I}$ of pairwise commuting smooth, normal and connected closed subgroups of $G$, each of them is quasi-simple, such that the product morphism
\[\prod_{i \in I} G_i \longrightarrow G\]
is a central isogeny. These $G_i$ are called the \emph{quasi-simple components} of $G$.
The isogeny $\prod_{i\in I} G_i \longrightarrow G$ induces an isomorphism of buildings \[\prod_{i\in I} \B(G_i,F)=\B(\prod_{i\in I} G_i, F) \xrightarrow{\sim} \B(G,F)\] and an isomorphism of flag varieties over $F$ \[\prod_{i\in I} \Par(G_i)=\Par(\prod_{i\in I} G_i)\xrightarrow{\sim} \Par(G).\]
Let $t$ be an $F$-type of $G$. For each $G_i$, through projection, it determines an  $F$-type $t_i$ of $G_i$. Then we get Berkovich maps $\theta_{t_i, F}$. 
We call the type $t$ \emph{non-degenerate} if each $t_i$ is non trivial, i.e. the component $\Par_{t_i}(G_i)$ is non trivial.
Note that this generalizes the notion of non-degenerate types in \cite{rtw2010} subsection 3.1, where it is defined only for $F$-rational types.  If $t_i$ is trivial, the corresponding map $\theta_{t_i,F}$ is also trivial: \[\B(G_i,F)\longrightarrow \{\ast\}.\]  Therefore, it is mild to restrict to non-degenerate types.  Moreover, the Berkovich map $\theta_{t, F}: \B(G,F)\lra \Par_t(G)^\Berk$ is an embedding  if the type $t$ is non-degenerate. Indeed, if $t$ is $F$-rational, this is \cite{rtw2010} Proposition 3.29.  In the general case, take a finite Galois extension $L|F$ to split $G$. Then for each $i$, there exists a non trivial $L$-type $t_{i, L}$ (which is thus rational) dominating $t_i$. Decomposing each $G_{i,L}$ and $t_{i,L}$ into quasi-simple components and using the base change diagram, one easily sees that $\theta_{t, F}$ is an injection.

\subsection{Berkovich compactifications of Bruhat-Tits buildings}
Now we discuss the Berkovich compactification of $\B(G,F)$ with respect to the map \[\theta=\theta_{\mu,F,E}: \B(G,F)\lra \Fl(G,\mu)^{\Berk}.\] Recall that we have the associated $F$-type $t=t_\mu$. 
We denote the image of the Berkovich map as \[\B_t(G,F)=\theta(\B(G,F))\subset \Fl(G,\mu)^{\Berk},\] since it is isomorphic to the image $\theta_{t,F}(\B(G,F))$ under the projection map $\tr{pr}_{E|F}$ as discussed above.
Let \[\overline{\B_t(G,F)}\] denote the closure of $\B_t(G,F)$ inside the flag variety $\Fl(G,\mu)^{\Berk}$.  It is also the closure\footnote{In \cite{rtw2010} there is an equivalent definition by taking the image of the map $G(F)\times \ov{A_t(S,F)}\ra \Par_t(G)^\Berk, (g,x)\mapsto gxg^{-1}$, endowed with the quotient topology, where $\ov{A_t(S,F)}$ is the closure of the image of some apartment $A(S,F)$ in   $\Par_t(G)^\Berk$.} of $\B_t(G,F)$ inside the flag variety $\Par_t(G)^{\Berk}$.
It will be a desired compactification.

An important property of the compactification is that it  extends  the base change functoriality:

\begin{prop}
\label{bc2}
\begin{enumerate}
	\item 

 Let $L|F$ be an extension of non-archimedean fields, then there exists a unique $G(F)$-equivariant continuous map $\overline{\B_t(G,F)}\longrightarrow \overline{\B_t(G,L)}$ making the following diagram commutes:

\[\xymatrix{
\B(G,F) \ar[r]^{i_{F,L}} \ar[d]^{\theta_{t,F}} &\B(G,L)\ar[d]^{\theta_{t,L}}\\
\overline{\B_t(G,F)} \ar[r] & \overline{\B_t(G,L)}}.\]

This map is a  homeomorphism onto its image.

\item The map $\theta_{t,F}: \B(G,F) \longrightarrow \overline{\B_t(G,F)}$  is continuous, open and its image $\B_t(G,F)$ is dense. 
\end{enumerate}
\end{prop}

We refer to section 2.4 of \cite{rtw2010} for more details.  
Although their work  is about $F$-rational type $t$, but these results indeed also hold for general types (use their arguments in section 2.4 and Appendix C). 
By section 4 of \cite{rtw2010} we can describe the compactification $\overline{\B_t(G,F)}$ through a stratification indexed by certain parabolic subgroups of $G$. Each boundary stratum is isomorphic to the Bruhat-Tits building of the Levi quotient of such a parabolic subgroup. To introduce these results, we need more notions. We follow the routine of \cite{rtw2010}, first assume $t=t_\mu$ is $F$-rational, and then deduce the general case.

We first introduce the notions of \emph{osculatory} and \emph{relevant} for parabolic groups, following \cite{rtw2010} subsection 3.2.
For a group scheme $H$ over a base $S$, two parabolic subgroups $P$ and $Q$ are called osculatory if their intersection $P \cap Q$ is still a parabolic subgroup. For our fixed reductive group $G/F$, an $F$-rational type $t$, and a parabolic subgroup $Q$ of $G$ over $F$, we can define the  \emph{osculatory variety} \[\Osc_t(Q),\] which is an algebraic variety over $F$ representing the following functor
\[(\textbf{Sch}/F)^{op} \longrightarrow \textbf{Sets},\quad S \mapsto \{P \in  \Par_t(G)(S)|\,  P \  is \ osculatory \  with \  Q \times_{F} S \}.\]   It is a closed subvariety of $\Par_t(G)$. On the other hand, the map $\Osc_t(Q)\longrightarrow \Par(Q)$ defined by \[\Osc_t(Q)(S)\longrightarrow \Par(Q)(S), \quad P \mapsto P \cap (Q\times_F S)\] is an isomorphism onto a connected component of $\Par(Q)$. Moreover, since the type $t$ is $F$-rational,  $\Osc_t(Q)(F)$ is non empty, thus the resulting type for $Q$ is still $F$-rational and for simplicity we still denote it by $t$. Let $Q_{ss}$ denote the reductive (or named Levi) quotient\footnote{Our notation $Q_{ss}$ is thus different from that in \cite{rtw2010}, where it means the semisimple quotient. Since we will use reduced Bruhat-Tits buildings throughout, this difference will disappear when passing to buildings.} of $Q$. Then $\Par(Q)=\Par(Q_{ss})$. In summary, we have  the following canonical maps: \[\Par_t(Q_{ss})=\Par_t(Q)\cong \Osc_t(Q)\hookrightarrow \Par_t(G).\]
Note that different $Q$ and $Q'$ can represent the same osculatory variety $\Osc_t(Q)$. The $F$-parabolic subgroup
$Q$ is called \emph{$t$-relevant} if it is maximal among all $F$-parabolic subgroups $Q'$ representing $\Osc_t(Q)$.

For later application, it is more convenient to use Levi subgroups. Take a Levi decomposition for $Q=N \rtimes M$. Pick up  an element $P \in \Osc_t(Q)(F)$ compatible with $M$ in the sense that \[M/(M \cap P)\cong Q/(Q \cap P)\] through the natural inclusion $M\subset Q$. Through this identification, we can further identify $\Par_{t}(M)$ with $\Par_t(Q_{ss})$ which is compatible with the isomorphism $M \cong Q_{ss}$.  Then we may replace $Q_{ss}$ by the Levi subgroup $M$ and we can rewrite the above canonical embedding as follows: \[\Par_t(M)\cong M/(M \cap P) \hookrightarrow G/P \cong \Par_t(G).\]

Now through such map, each Bruhat-Tits building $\B(Q_{ss},F)$ (or equivalently $\B(M,F)$) will contribute to the compactification of $\B(G,F)$ in the following way:
\[\B(Q_{ss},F)\xrightarrow{\theta_{t,F}}\Par_t(Q_{ss})^{\Berk}=\Par_t(Q)^{\Berk}\hookrightarrow \Par_t(G)^{\Berk}.\]
The image of such map is contained in $\overline{\B_t(G,F)}$ and the compactification consists of such contributions:
 
\begin{prop}[\cite{rtw2010} Theorem 4.1, Propositions 4.5, 4.6]
	\label{prop compactify rational}
	Assume that the type $t$ is $F$-rational.
\begin{enumerate}
	\item 

 We have the following stratification (each contribution is a locally closed subspace):

\[\overline{\B_t(G,F)}= \coprod_{\substack{Q\in \Par(G)(F)\\ Q\; t\tr{-relevant}}} \B(Q_{ss},F).\]

\item The natural $G(F)$-action on $\B(G,F)$ extends uniquely to an action on $\overline{\B_t(G,F)}$ with \[g \B(Q_{ss},F)=\B(gQ_{ss}g^{-1},F).\]

\item (Base change functoriality) For any non-archimedean extension $L|F$, the following diagram is commutative:

\[\xymatrix{
\B(Q_{ss},F) \ar[r]^{i_{F,L}} \ar[d] &\B(Q_{ss},L)\ar[d]\\
\overline{\B_t(G,F)} \ar[r] & \overline{\B_t(G,L)}}\]
\end{enumerate}
\end{prop}
We may also write the compatification as a union over all $F$-rational parabolic subgroups
\[ \overline{\B_t(G,F)}= \bigcup_{Q\in \Par(G)(F)} \B(Q_{ss},F).\]

Now we turn to the general case that $t=t_\mu$ is not necessary $F$-rational. As in Appendix C of \cite{rtw2010}, we can find an $F$-rational type $\tau$ uniquely determined by $t$.  
More precisely, the type $\tau$ is constructed as follows. Take a minimal $F$-parabolic subgroup $P_0$ of $G$, then we get a closed and smooth subscheme $\mathrm{Osc}_t(P_0)$ of $\Par_t(G)$. There exists  a largest $F$-parabolic subgroup \[Q_0\subset G\] stabilizing $\mathrm{Osc}_t(P_0)$. The conjugacy class of $Q_0$ does not depend on the choice of $P_0$, which defines an $F$-rational type $\tau$. If $t$ is $F$-rational, then $\tau=t$.
If $G$ is quasi-split, then $\tau$ is the largest $F$-rational type dominated by $t$. For example, if $G=U(3)$ is the quasi-split unitary group over $F$ defined by a Hermitian vector space $V=E^3$ with respect to a quadratic unramified extension $E|F$, then $G_{E}=\GL_3$. Both $E$-types defined by minuscule cocharacters $\mu_1=(1,0,0)$ and $\mu_2=(1,1,0)$ will correspond to a single $F$-type $t$, and the resulting $\tau$ will correspond to the conjugacy class of Borel subgroups of $G$.

We have the following result, generalizing Proposition \ref{prop compactify rational}, cf. \cite{rtw2010} Appendix C.
\begin{prop}
\label{levistrata}
We have the following decomposition of the Berkovich compactification
\[\overline{\B_{t}(G,F)}= \coprod_{\substack{Q\,\tr{standard}\\ Q\,\tau\tr{-relevant}}} \coprod_{g \in G(F)}g\Big(\B_\tau(Q_{ss},F)\Big),\]
where $Q$ runs through the set of standard $F$-rational parabolic subgroups of $G$ which are $\tau$-relevant.
\end{prop}
\begin{proof}
For later applications, during the proof,
we reformulate the results of \cite{rtw2010} in terms of Levi subgroups. 

First, notice that even if $t$ is not $F$-rational, the Berkovich compactification still has the base change functoriality. Take any non-archimedean field extension $L|E$, through the Berkovich maps, the closure of the Bruhat-Tits building are the same, cf. Proposition \ref{bc2}.
Thus we can apply such base change technique. Now assume $L|E $ is a finite Galois extension that splits $G$. Let $t_B$ denote the $L$-type corresponding to Borel subgroups of $G_L$. Pick up an $L$-type $t_1$ that corresponds to the connected component $\Fl(G,\mu)_L$ of $\Par_t(G)_L$. Let $\tau$ be the $F$-rational type attached to $t$ as above.
We have the following two natural projection maps among Berkovich flag varieties:
\[\xymatrix{
	& \Par_{t_B}(G_L)^{\Berk} \ar[ld]  \ar[rd] \\
	\Par_{\tau}(G)_{L}^{\Berk} &  &\Par_{t_1}(G_L)^{\Berk}.}\] 
Note that $\Par_{t_1}(G_L)=\Fl(G,\mu)_L$.
Through the Berkovich maps and taking the closure of the Bruhat-Tits building inside these flag varieties, we get 
\[\overline{\B_t(G,F)}\simeq \overline{\B_{\tau}(G,F)}\]  and
the following commutative diagram (this is also used in Appendix C of \cite{rtw2010} in a similar way):
\[\xymatrix{
	& \overline{\B_{t_B}(G,F)} \ar[ld] \ar[rd] \\
	\overline{\B_{\tau}(G,F)}  & \cong   &\overline{\B_t(G,F)}.}\] 

On the other hand, for any $F$-parabolic subgroup $Q$ of $G$ and take a Levi subgroup $M$ for $Q$, through the osculatory  variety $\Osc_{t_{B}}(Q_L)$, we enlarge the diagram into the following one: 
\[\xymatrix{
	\B(M,F)\ar[r] & \Par_{t_B}(M_L)^{\Berk} \ar[r] & \Par_{t_B}(G_L)^{\Berk} \ar[ld] \ar[rd]& \\
	&\Par_{\tau}(G)_{L}^{\Berk} & & \Par_{t_1}(G_L)^{\Berk}}\]
By Proposition \ref{prop compactify rational}, $\B(M,F)\cong \B(Q_{ss}, F)$ contributes to the compactification $\overline{\B_{\tau}(G,F)}$. Let $Q$ run over $F$-parabolic subgroups for $G$. The union of these contributions is the whole space $\overline{\B_{\tau}(G,F)}$. Through the isomorphism between these two compactifications, we see that $\overline{\B_t(G,F)}$ is also a union of such contributions. 

We would like to see these contributions inside the $p$-adic flag variety $\Par_{t_1}(G_L)^{\Berk}=\Fl(G,\mu)_L^\Berk$ in a natural way.
To this end, notice that there exists a commutative diagram:
\[\xymatrix{
	\Par_{t_B}(M_L)^{\Berk} \ar[d] \ar[r] & \Par_{t_B}(G_L)^{\Berk} \ar[d]\\
	\Par_{t_1}(M_L)^{\Berk} \ar[r] & \Par_{t_1}(G_L)^{\Berk}.}\]
Therefore, we may rewrite the contribution from the Bruhat-Tits building $\B(M,F)$  to the type $t$ compactification $\overline{\B_t(G,F)}$ through the following map \[\B(M,F) \xrightarrow{\theta_{t_1,F}} \Par_{t_1}(M_L)^{\Berk} \hookrightarrow\Par_{t_1}(G_L)^{\Berk},\] and denote its image by $\B_{\tau}(M,F)$.  

To get disjoint unions, one just considers the $\tau$-relevant parabolic subgroups by Proposition \ref{prop compactify rational}. We have thus verified the proposition.
\end{proof}
In the above proof,
the contributions of $M$ are conjugated by $G(F)$-action. Therefore the locus \[\bigcup_{g \in G(F)}g\Big(\B_\tau(M,F)\Big)\]  only relies on the conjugacy class of $M$, independent of choices of parabolic subgroups $Q$ containing $M$.
Thus, we may also write the compactification as a union over all standard $F$-rational Levi subgroups
\[ \overline{\B_{t}(G,F)}= \bigcup_{M\,\tr{standard}}\bigcup_{g \in G(F)}g\Big(\B_\tau(M,F)\Big).\]


\section{Bruhat-Tits buildings and $p$-adic period domains}
\label{newtonstrata}

From now on, we will always work with a \emph{minuscule} cocharacter $\mu: \G_m\ra G_{\ov{F}}$ (and its conjugacy class $\{\mu\}$), and discuss the basic relation between Bruhat-Tits buildings and $p$-adic period domains. 

\subsection{Berkovich spaces and adic spaces}
Since $p$-adic period domains are defined more conveniently using the theory of perfectoid spaces and diamonds (cf. \cite{sw2020} and \cite{Sch17}), we first briefly review the basic comparison between Berkovich analytic spaces (\cite{Ber90, Ber93}) and adic spaces (\cite{Hub96}).

Let $k$ be a complete non-archimedean field with residue field of characteristic $p$. Recall that an adic space $X$ over $k$ is called locally of finite type if it is locally of the form $\Spa(R, R^+)$, where $R^+=R^\circ$ and $R$ is a quotient of the algebra $k\lan T_1,\cdots, T_n\ran$ for some $n$. Such an affinoid $k$-algebra $(R,R^+)$ is called topologically of finite type. The adic space $X$ is called taut if it is quasi-separated and for any quasi-compact open subset $U\subset X$, the closure $\ov{U}$ of $U$ in $X$ is also quasi-compact. For example, if $X$ is partially proper over $k$ (cf. \cite{Hub96} Definition 1.3.3), then it is taut.
\begin{thm}[\cite{Hub96} Proposition 8.3.1 and Lemma 8.1.8]\label{thm Berk vs adic}
There is an equivalence  between 
\begin{itemize}
	\item 
the category of Hausdorff strictly $k$-analytic Berkovich spaces, and 
\item the category of taut adic spaces which are locally of finite type over $k$.
\end{itemize}
If $X^\Berk$ is mapped to $X^\ad$ under this equivalence, then there is an injective map of sets $X^\Berk \hookrightarrow X^\ad$ with image consisting of the subset of rank 1 valuation points. This map is not continuous in general, but there is a continuous retraction $X^\ad\ra X^\Berk$, identifying $X^\Berk$ as the maximal Hausdorff quotient of $X^\ad$.
\end{thm}

The underlying topological space of the  adic space $X^\ad$ is locally spectral. Recall that for a locally spectral space $Y$, its dimension $\dim\,Y$ is the supremum of the lengths $n$ of the chains of specializations \[x_0\succ x_1\succ \cdots \succ x_n\] of points of $X$, where the partial order $\succ$ is defined as $x\succ y\Leftrightarrow y\in \ov{\{x\}}$.
The underlying topological space of $X^\Berk$ is locally compact and Hausdorff. Its dimension $\dim X^\Berk$ is defined as the supremum of the Krull dimension $\dim \,R$ of the open affinoid subspaces $\M(R)\subset X^\Berk$. Then we have $\dim \,X^\ad=\dim\, X^\Berk$ under the equivalence of Theorem \ref{thm Berk vs adic}. 

Let $x\in X^\Berk$ be a point. We will also view it as a point of $X^\ad$ under the above inclusion $X^\Berk\subset X^\ad$. Consider the closure $\ov{\{x\}}\subset X^\ad$ of the subset $\{x\}\subset X^\ad$.  By definition, this is the set of specializations of $x$ in $X^\ad$.
Then we have $\ov{\{x\}}\subset \pi^{-1}(x)$, where $\pi: X^\ad\ra X^\Berk$ is the quotient map.

If $X$ is an algebraic variety (as usual, we refer to a scheme which is separated and locally of finite type) over $k$. Then we have the associated Berkovich space $X^\Berk$ (which is a Hausdorff strictly $k$-analytic space, cf. \cite{Ber90}) and adic space $X^\ad$ (which is taut and locally of finite type over $\Spa\,k$, cf. \cite{Hub96}). Then $X^\Berk$ corresponds to $X^\ad$ under the equivalence functor in Theorem \ref{thm Berk vs adic}.  If $U\subset X^\ad$ is a partially proper (cf. \cite{Hub96} Definition 1.3.3) open subspace, then $U$ is taut, and the associated Hausdorff strictly $k$-analytic space $U^\Berk$ under Theorem \ref{thm Berk vs adic} is given by the maximal Hausdorff quotient of $U$, which is an open subspace of $X^\Berk$.

Occasionally we will talk about diamonds, so we make a very brief review on the comparison between analytic adic spaces and diamonds, for details see \cite{Sch17, sw2020}.  Let $\Perf$ be the category of perfectoid spaces in characteristic $p$. There are two natural Grothendieck topologies on it: the pro-\'etale topology and the $v$-topology, cf. \cite{Sch17} section 8. By definition, a diamond is a pro-\'etale sheaf on $\Perf$ which can be written as a quotient $X/R$, with $X\in \Perf$ and $R$ a representable pro-\'etale relation, cf. \cite{Sch17} Definition 11.1. For any analytic adic space $X$ (see \cite{Hub96} page 39 or \cite{sw2020} subsection 4.3) over $\Spa\,\Z_p$, one can associate it a diamond $X^\Diamond$ as follows. For any $T\in \Perf$, $X^\Diamond(T)$ is the set of isomorphism classes of pairs $(T^\sharp, T^\sharp\ra X)$, where $T^\sharp$ is an untilt of $T$ over $\Z_p$, and $T^\sharp\ra X$ is a morphism of adic spaces over $\Z_p$. For any diamond $Y$, one has well defined notions of underlying topological space $|Y|$ and \'etale site $Y_{\tr{et}}$. If $Y=X^\Diamond$ for an analytic adic space $X$ over $\Spa\,\Z_p$, then we have (cf. \cite{Sch17} Lemma 15.6)
\[|X^\Diamond|\cong |X|, \quad X^\Diamond_{\tr{et}}\cong X_{\tr{et}}. \]
Moreover, $X^\Diamond$ is locally spatial in the sense of \cite{Sch17} Definition 11.17.  These discussions on analytic adic spaces over $\Z_p$ apply in particular to the above category of taut adic spaces which are locally of finite type over a complete non-archimedean field $k$ with residue field of characteristic $p$.

\subsection{Newton and Harder-Narasimhan stratifications of $p$-adic flag varieties}\label{subsection Newton and HN}
We first review two natural stratifications on the $p$-adic flag variety $\Fl(G,\mu)^{\ad}$ (or $\Fl(G,\mu)^{\Berk}$) in $p$-adic Hodge theory: Newton and Harder-Narasimhan stratifications.

Let $B(G)$ be the set of $\sigma$-conjugacy classes in $G(\breve{F})$. It admits an explanation as the set of isomorphism classes of isocrytals with $G$-structure (cf. \cite{kott1985} and \cite{RR}). There are two invariants attached to a $\sigma$-conjugacy class, given by the Newton map (cf. \cite{kott1985} section 4)
\[\nu: B(G)\lra \N(G)\]
and the Kottwitz map (cf. \cite{kott1997} 4.9 and 7.5, \cite{RR} Theorem 1.15)
\[\kappa: B(G)\lra \pi_1(G)_\Gamma.\]
The induced map \[(\nu,\kappa): B(G)\lra \N(G)\times \pi_1(G)_\Gamma\] is injective (cf. \cite{kott1997} 4.13).  On $B(G)$ we have a partial order: $b\leq b' \Leftrightarrow \nu(b)\leq \nu(b')$.
Let $B(G)_{basic}\subset B(G)$ be the subset of basic elements (cf. \cite{kott1985} section 5), which consists of those $b\in B(G)$ such that $\nu(b)$ factors through the center of $G$. Then the restriction of $\kappa$ induces a bijection (cf. \cite{kott1985} Proposition 5.6) \[\kappa: B(G)_{basic}\st{\sim}{\lra} \pi_1(G)_\Gamma.\]
We will view $B(G)$ as the set of isomorphism classes of $G$-bundles on the Fargues-Fontaine curve. More precisely, let $C|F$ be an algebraically closed perfectoid field, with tilt $C^\flat$. We get the associated Fargues-Fontaine curve $X=X_{C^\flat, F}$ over $F$, which can be viewed as either a one dimensional Noetherian scheme over $\Spec\,F$ or an adic space over $\Spa\,F$.
Fargues proved that there is a bijection of pointed sets (cf. \cite{cfs2021} Theorem 1.4):
\[B(G)\st{\sim}{\lra} H^1_{\tr{et}}(X, G),\quad b\mapsto \mathcal{E}_b. \]
Moreover, the Newton map and Kottwitz map both admit geometric interpretations in terms of $G$-bundles, for more details see \cite{cfs2021} Theorem 1.10.

Attached to the pair $(G, \{\mu^{-1}\})$, we have
 the Kottwitz set (cf. \cite{kott1997} section 6):  \[B(G,\mu^{-1})=\{b \in B(G)\,|\, v(b)\leq \mu^{-1,\diamond},\ \kappa (b)=\mu^{-1,\sharp}\},\] here we are using Galois average \[\mu^{-1,\diamond}=[\Gamma:\Gamma_{\mu^{-1}}]^{-1}\sum_{\tau \in \Gamma/\Gamma_{\mu^{-1}}}\mu^{-1,\tau}\in \N(G),\] and $\mu^{-1,\sharp}\in \pi_1(G)_{\Gamma}$. This is a  finite subset of $B(G)$ which contains a unique basic element (cf. \cite{kott1997} 6.4), and it only depends on the conjugacy class $\{\mu^{-1}\}$. We will also use the generalized Kottwitz set introduced in \cite{cfs2021}: for any $\epsilon\in \pi_1(G)_\Gamma$ and $\delta\in X_\ast(A)_\Q^+$ such that $\epsilon\equiv\delta$ in $\pi_1(G)_\Gamma\otimes\Q$, we set
 \[B(G,\epsilon, \delta)=\{b\in B(G)\,|\, \kappa(b)=\epsilon, \nu(b)\leq \delta\}. \]
 Then by definition we have \[B(G,\mu^{-1})=B(G, \mu^{-1,\sharp}, \mu^{-1,\diamond}).\]

A key property of the Kottwitz set is that we can pass to the adjoint group (see \cite{kott1997} 4.11 and 6.5): the natural map $G\ra G_{ad}$ induces a bijection \[B(G,\epsilon,\delta)\cong B(G_{ad},\epsilon_{ad}, \delta_{ad}).\] This property is suitable for our application, since  some geometric objects like the flag varieties and the Bruhat-Tits building only depend on $G_{ad}$. 


Now, we introduce the Newton stratification on the $p$-adic flag variety $\Fl(G,\mu)^{\ad}$ (and $\Fl(G,\mu)^{\Berk}$). It will be indexed by the Kottwitz set $B(G,\mu^{-1})$ (be careful there is a sign change). Roughly speaking, this is defined by modifying the trivial $G$-bundle over the Fargues-Fontaine curve. In particular, this stratification is stable under the natural $G(F)$-action on $\Fl(G,\mu)^\ad$,  as $G(F)$ is the automorphism group of the trivial $G$-bundle. More precisely, the construction actually involves the $B_{\dR}^+$-affine Grassmannian $\Gr_G^{B^+_{\dR}}$ introduced in \cite{sw2020} and the $v$-stack of $G$-bundles $\Bun_G$ on the Fargues-Fontaine curve  introduced in \cite{FS}. The $\Gr_G^{B^+_{\dR}}$ is a small $v$-sheaf over $\tr{Spd}\,F=(\Spa\,F)^\Diamond$. For each geometric conjugacy class $\{\mu\}$ of cocharacters of $G$, using Cartan decomposition one can define a sub $v$-sheaf $\Gr_{G,\mu}^{B_{\dR}^{+}}\subset \Gr_{G,E}^{B^+_{\dR}}$ over $\tr{Spd}\,E$, where $E=E(G, \{\mu\})$, which is the affine Schubert cell in this setting. See \cite{sw2020} Lecture 19 for more details. Back to our minuscule $\mu$,
the following $p$-adic  Bialynicki-Birula map (cf. \cite{cs2017} Proposition 3.4.3) from the affine Schubert cell attached to $\{\mu\}$  to the diamond flag variety \[\pi_{G,\mu}^{BB}: \Gr_{G,\mu}^{B_{\dR}^{+}}\longrightarrow \Fl(G,\mu)^{\ad,\Diamond}\] is an isomorphism, cf. \cite{cs2017} Theorem 3.4.5. On the other hand, there is a natural Beauville-Laszlo map \[BL_1: \Gr_{G}^{B_{\dR}^{+}} \longrightarrow \Bun_{G} \] relating $B_{\dR}^{+}$ affine Grassmannian with $\Bun_{G}$ by modifying the trivial $G$-bundle, cf. \cite{FS} III.3. By Fargues's theorem on the classification of $G$-bundles over the Fargues-Fontaine curve, we have a natural identification \[|\Bun_G| \cong B(G) \] as sets (cf. \cite{FS} Theorem III.2.2); indeed this is also a topological isomorphism with the order topology on $B(G)$ due to \cite{eva2024}. Combining these maps together, we get the desired morphism \[\Fl(G,\mu)^{\ad,\Diamond}\longrightarrow \Bun_G.\]
On topological spaces, we get a map \[|\Fl(G,\mu)^{\ad}|\lra B(G),\] called the Newton map. Explicitly, for any point $x\in \Fl(G,\mu)^{\ad}(C,C^+)$ with $C|E$ algebraically closed perfectoid field, by  Bialynicki-Birula and Beauville-Laszlo maps, we get a $G$-bundle $\mathcal{E}_{1,x}$ on the Fargues-Fontaine curve $X_{C^\flat, C^{+,\flat}}$ by modifying the trivial $G$-bundle using $x$ at the canonical point $\infty \in X_{C^\flat, C^{+,\flat}}$ corresponding the untilt $C$ of $C^\flat$. The isomorphism class of the $G$-bundle $\mathcal{E}_{1,x}$  defines an element  \[b(\mathcal{E}_{1,x})\in B(G)\] by Fargues's theorem. This is the point-wise description of the Newton map.
As in \cite{cs2017} Proposition 3.5.7 and the paragraph above  there,
 this map is semi-continuous (based on the result of Kedlaya-Liu \cite{KL} in the case $G=\GL_n$). Moreover,  it factors through the maximal Hausdorff quotient, which is exactly the topological space \[|\Fl(G,\mu)^{\Berk}|\] underlying the Berkovich space $\Fl(G,\mu)^{\Berk}$ by Theorem \ref{thm Berk vs adic}. By \cite{cs2017} Proposition 3.5.3, the image of the Newton map is contained in $B(G,\mu^{-1})$. For any $b \in B(G,\mu^{-1})$, let $\Fl(G,\mu)^{\ad, b}$ denote the fiber over $b$, then this is a locally closed subspace of $\Fl(G,\mu)^{\ad}$. In this way we get the Newton stratification for the adic space \[\Fl(G,\mu)^{\ad}=\coprod_{b\in B(G,\mu^{-1})}\Fl(G,\mu)^{\ad, b}.\]
 This stratification is invariant under the natural $G(F)$-action on $\Fl(G,\mu)^{\ad}$.
 We summarize the main properties of the Newton stratification  as follows.
 \begin{prop}\label{prop Newton}
 \begin{enumerate}
 	\item For each $b\in B(G,\mu^{-1})$, the corresponding stratum $\Fl(G,\mu)^{\ad, b}$  is non empty, in other words the image of $BL_1$ is exactly $B(G,\mu^{-1})$.
 	\item For each $b\in B(G,\mu^{-1})$, the dimension of the stratum (as a locally spectral space) is \[\dim\,\Fl(G,\mu)^{\ad, b}=\dim \,\Fl(G,\mu)-\lan\nu(b), 2\rho\ran=\lan \mu-\nu(b), 2\rho\ran,\] where $\rho$ is the half of positive absolute roots of $G$.
 	\item For each $b\in B(G,\mu^{-1})$, we have the closure relation
 	\[ \ov{\Fl(G,\mu)^{\ad, b}}=\coprod_{b'\geq b}\Fl(G,\mu)^{\ad, b'},\]
 	where $b'$ runs through the element $b'\in B(G,\mu^{-1})$ such that $b'\geq b$ for the partial order $\geq$.
 \end{enumerate}
 \end{prop}
\begin{proof}
(1) is due to Rapoport, see \cite{cs2017} Remark 3.5.8 or \cite{shen2023} Proposition 6.7.

(2) was proved in special case in \cite{cs2017} Proposition 4.2.23, and in general case in \cite{cfs2021} Proposition 5.3 and \cite{shen2023} Proposition 3.1.

(3) was proved in \cite{eva2024} Corollary 6.9.
\end{proof} 
 For each $b\in B(G,\mu^{-1})$, we get a corresponding locally closed subspace $\Fl(G,\mu)^{\Berk,b}\subset \Fl(G,\mu)^{\Berk}$ 
 and a $G(F)$-invariant stratification of the Berkovich space \[\Fl(G,\mu)^{\Berk}=\coprod_{b\in B(G,\mu^{-1})}\Fl(G,\mu)^{\Berk,b}.\]
 Recall that
 the Kottwitz set $B(G,\mu^{-1})$ contains a unique basic element $b_0$, which is minimal with respect to the partial order $\leq$. The corresponding basic stratum $\Fl(G,\mu)^{\ad,b_0}$ (resp. $\Fl(G,\mu)^{\Berk,b_0}$)  is  open in $\Fl(G,\mu)^{\ad}$ (resp. $\Fl(G,\mu)^{\Berk}$). Moreover, $\Fl(G,\mu)^{\Berk,b_0}$ is the Hausdorff strictly Berkovich space corresponding to the adic space $\Fl(G,\mu)^{\ad,b_0}$  by Theorem \ref{thm Berk vs adic}.
 We call both \[\Fl(G,\mu)^{\ad,b_0}\quad \tr{and}\quad \Fl(G,\mu)^{\Berk,b_0}\] the $p$-adic period domains\footnote{They arise as the target of the Hodge-Tate period map for the local Shimura variety attached to the triple $(G, \{\mu^{-1}\}, b_0)$ at infinite level.
 	In \cite{eva2024} this is called the admissible locus, although this is slightly different from the original terminology in the de Rham setting as in \cite{RZ96} or \cite{sw2020}. In section 6, we will translate back to the de Rham setting by introducing the dual local Shimura datum.} attached to $(G, \{\mu^{-1}\}, b_0)$. By Proposition \ref{prop Newton} (3), the subspace $\Fl(G,\mu)^{\Berk,b_0}$ is dense in $\Fl(G,\mu)^{\Berk}$.
 	
 The Harder-Narasimhan stratification on $\Fl(G,\mu)^{\ad}$ is based on a semi-continuous function
 \[\HN: |\Fl(G,\mu)^{\ad}|\lra B(G,\mu^{-1}),\]
 which can be constructed by\footnote{For both versions, as usual we need the Tannakian formalism to deal with general reductive groups.} either the theory of Harder-Narasimhan filtrations of filtered vector spaces as in \cite{DOR10}, or the theory of Harder-Narasimhan filtrations of  admissible modifications of vector bundles on the Fargues-Fontaine curve as in \cite{shen2023} section 3. Using the later approach, for  any $x\in \Fl(G,\mu)^{\ad}(C,C^+)$ with $C|E$ algebraically closed perfectoid field, $\HN(x)$ is given by the HN vector of the modification triple $(\mathcal{E}_1, \mathcal{E}_{1,x}, f_x)$, see \cite{shen2023} for more details.  In particular, we get a stratification
 \[\Fl(G,\mu)^{\ad}=\coprod_{b\in B(G,\mu^{-1})}\Fl(G,\mu)^{\ad, \HN=b}, \]
 with each stratum $\Fl(G,\mu)^{\ad, \HN=b}$ locally closed in $\Fl(G,\mu)^{\ad}$. Moreover, this stratification is also $G(F)$-invariant.
 Unfortunately, we do not know enough information on this stratification like Proposition \ref{prop Newton} (neither the non emptiness of $\Fl(G,\mu)^{\ad, \HN=b}$, nor dimension formula, nor closure relation, for example see \cite{shen2023} for more discussions). On the other hand, there is a unique open stratum $\Fl(G,\mu)^{\ad, \HN=b_0}$. All we will need later is about this stratum, which is relatively easier to study.
 \begin{prop}\label{prop HN}
 We have a $G(F)$-equivariant inclusion
 \[\Fl(G,\mu)^{\ad, b_0}\subset \Fl(G,\mu)^{\ad, \HN=b_0}.\]
 In particular, the open Harder-Narasimhan stratum $\Fl(G,\mu)^{\ad, \HN=b_0}$ is dense in the adic flag variety $\Fl(G,\mu)^{\ad}$.
 \end{prop}
 \begin{proof}
 	See \cite{shen2023} subsection 3.3 or \cite{eva2024} Lemma 4.7.
 \end{proof}
 From the description of the Harder-Narasimhan stratification in terms of admissible modifications of $G$-bundles on the Fargues-Fontaine curve, together with the fact that for any $S=\Spa(R, R^+)\in \Perf$, the categories of $G$-bundles on the relative Fargues-Fontaine curves $X_{R, R^+}$ and $X_{R, R^\circ}$ respectively are equivalent (also deduced by the work of Kedlaya-Liu \cite{KL}),  the map $\HN$ factors through the Berkovich flag variety $\Fl(G,\mu)^{\Berk}$.  Therefore, we get a $G(F)$-invariant stratification \[\Fl(G,\mu)^{\Berk}=\coprod_{b\in B(G,\mu^{-1})}\Fl(G,\mu)^{\Berk, \HN=b},\] with each stratum $\Fl(G,\mu)^{\Berk, \HN=b}$ locally closed. The stratum $\Fl(G,\mu)^{\Berk, \HN=b_0}$ attached to $b_0$ is open, which corresponds to the open adic stratum $\Fl(G,\mu)^{\ad, \HN=b_0}$ under Theorem \ref{thm Berk vs adic}.
 By Proposition \ref{prop HN} we have a $G(F)$-equivariant inclusion of open subspaces of $\Fl(G,\mu)^\Berk$ \[\Fl(G,\mu)^{\Berk, b_0}\subset \Fl(G,\mu)^{\Berk, \HN=b_0}.\]
 
 We call \[\Fl(G,\mu)^{\ad, \HN=b_0}\quad (\tr{resp.}\quad \Fl(G,\mu)^{\Berk, \HN=b_0})\] the semistable locus\footnote{Our terminology is compatible with that in \cite{DOR10}. In \cite{eva2024} it is called weakly admissible locus, although the setting is slightly different from the original de Rham setting as in \cite{RZ96} Chapter 1 or \cite{DOR10} Part 3. Again, in section 6 we will translate back to the de Rham setting by introducing the dual local Shimura datum.} in  $\Fl(G,\mu)^{\ad}$ (resp. $\Fl(G,\mu)^{\Berk}$). By \cite{DOR10} XI.1 (see also \cite{RZ96} proof of Proposition 1.36), we can describe the complement:
 \[\Fl(G,\mu)^{\ad}\setminus  \Fl(G,\mu)^{\ad, \HN=b_0}=\bigcup_{i\in I}G(F)Z_i,\] where $(Z_i)_{i\in I}$ is a finite collection of Zariski closed Schubert varieties. Similarly for the Berkovich version.
 
\subsection{Embeddings of Bruhat-Tits buildings into $p$-adic period domains}

After the above preparations, we can discuss the relation between the Berkovich map \[\theta: \B(G,F)\lra \Fl(G,\mu)^{\Berk} \]
and the Newton stratification \[\Fl(G,\mu)^{\Berk}=\coprod_{b\in B(G,\mu^{-1})}\Fl(G,\mu)^{\Berk,b}.\]
As before we denote $\B_t(G,F)=\theta(\B(G,F))$.
 As said in section \ref{ber}, the main intuition is that the image of the Berkovich map consists of \emph{very generic} points. As a warm up, we first deduce the following theorem concerning the semistable locus.

\begin{thm}\label{thm BT vs ss}
Through the Berkovich map, the Bruhat-Tits building $\B(G,F)$ is mapped into the semistable locus, i.e. $\B_t(G,F)\subset \Fl(G,\mu)^{\Berk, \HN=b_0}$.
\label{weak}
\end{thm}

\begin{proof}

Take a finite field extension $L|F$ that splits the group $G$. In fact we will prove a stronger result: $\theta_L(\B(G,L))$ is contained in the $\Fl(G,\mu)^{\Berk, \HN=b_0}_L$.  

Recall Proposition \ref{explicit} which computes the Berkovich map explicitly. We still use those notions with replacing the base field by $L$.
By Proposition \ref{explicit}, each point $x$ inside the image of the apartment is in fact a norm (instead of merely semi-norm) on the polynomial ring. Therefore under the natural map $\Omega(T,P)^{\Berk} \longrightarrow \Omega(T,P)$ (taking support), the point $x$ will correspond to the generic point of that variety. 

In the following we will pass to adic spaces. Consider the continuous quotient map $\Omega(T,P)^{\ad} \ra \Omega(T,P)^{\Berk} $ and the (discontinuous) inclusion $\Omega(T,P)^{\Berk}\subset \Omega(T,P)^{\ad} $.
We have the following commutative diagram of continuous maps between locally spectral spaces:

\[\xymatrix{
\Omega(T,P)^{\ad} \ar[d] \ar@{^{(}->}[r] & \Fl(G_L, \mu)^{\ad} \ar[d] \\
\Omega(T,P) \ar@{^{(}->}[r] & \Fl(G_L,\mu).
}\]
Both horizontal maps are open embeddings. In particular, the  map of the bottom line will send the generic point of $\Omega(T,P)$ to the generic point of $\Fl(G_L,\mu)$. Therefore the point $x$ is \emph{generic} in the adic space $\Fl(G_L,\mu)^{\ad}$, in the sense that it is in the fiber of the generic point of $\Fl(G_L,\mu)$ under the projection $\Fl(G_L,\mu)^{\ad}\ra \Fl(G_L,\mu)$. In particular,
it can not be contained in any proper Zariski closed subspace. On the other hand, the semistable locus $\Fl(G_L,\mu)^{\ad, \HN=b_0}$  is an open subspace of $\Fl(G_L,\mu)^{\ad}$ whose complement is  a profinite union of proper Zariski closed subspaces (some  Schubert varieties). Therefore $x$ lies in such locus. 

Passing back to Berkovich spaces, we get $x\in\Fl(G_L,\mu)^{\Berk, \HN=b_0}$.  Notice that the Bruhat-Tits building $\B(G,L)$ is a union of apartments. Therefore it lies in this $L$-semistable locus.
We can view $\B(G,F)$ as a subspace of $\B(G,L)$, which is mapped into the $L$-semistable locus  $\Fl(G,\mu)^{\Berk, \HN=b_0}_{L}$.  By the base change functoriality of the map $\theta$ (cf. Proposition \ref{bc1}), we have $\B_t(G,F)=\theta(\B(G,F))\subset \Fl(G,\mu)^{\Berk, \HN=b_0}$.
\end{proof}

Next, we will prove a stronger theorem, showing that the image $\B_t(G,F)=\theta(\B(G,F))$ in fact totally  lies in the $p$-adic period domain $\Fl(G,\mu)^{\Berk, b_0}$. This relies on some dimension arguments and basic geometric properties of the Newton stratification as in Proposition \ref{prop Newton}.
 
\begin{thm}\label{thm BT vs p-adic period}
Through the Berkovich map, the Bruhat-Tits building $\B(G,F)$ is mapped in to the $p$-adic period domain, i.e. $\B_t(G,F)\subset \Fl(G,\mu)^{\Berk, b_0}$.
\label{main1}
\end{thm}

\begin{proof}
 The first step is still applying the base change technique, cf. Proposition \ref{bc1}. For any point $o$ of the building $\B(G,F)$, take a suitable non-archimedean field extension $L|F$ such that $G_L$ is split over $L$ and through the natural map $\B(G,F) \hookrightarrow \B(G,L)$, the point $o$ becomes a special vertex for the later building. In practice, any complete non-archimedean field containing $\bar{F}$ with valuation group being the whole real numbers $\mathbb{R}$ will satisfy this requirement.

Again we apply the explicit formula of Proposition \ref{explicit}. Then the special vertex $o$ will correspond to the standard Gauss point of the Berkovich space \[\Omega(T,P)^{\Berk} \cong \A^{N,\Berk},\] here $N$ is the cardinality of $\Psi$ and it also equals to the dimension of $\Fl(G,\mu)$.
Inside the Berkovich space, the Gauss point (still denote it by $o$) corresponds to the following multiplicative norm: \[L[(X_{\alpha})_{\alpha \in \Psi}]\rightarrow \mathbb{R}_{\geq 0},\quad   \sum_{v \in \mathbb{N}^{\Psi}}a_v X^{v}\mapsto \max_{v} |a_v| \]  Its valuation residue field is $\widetilde{L}(x_{\alpha})$, here $\widetilde{L}$ is the residue field for $L$ and $\alpha$ run over roots in $\Psi$.  It is a field extension of $\widetilde{L}$ with transcendence degree $N$.

Next consider the quotient map $\pi: \Omega(T,P)^{\ad}\rightarrow \Omega(T,P)^{\Berk}$, the Gauss point $o$ has a distinguish lift inside the adic space. We can find this distinguish lift (denoted by $o_1$) by  the (discontinuous) inclusion $\Omega(T,P)^{\Berk}\subset \Omega(T,P)^{\ad}$. Or more concretely,we can  define the Gauss point for the $N$-dimensional disk in a similar way like \cite{sch2012} (which is about one dimensional case),  and view it as a point for the affine $\A^{N,\ad}$ adic space. Then this Gauss point is the desired distinguish lift.

From the discussion above, the valuation residue field of $o_1$ is unchanged, still a field extension of $\widetilde{L}$ with transcendence degree $N$. Then by the dimension formula in Lemma 3.2.2 of \cite{ext2022}, we know that the dimension of the spectral space $\overline{\{o_1\}}$ (closure of $o_1$ inside $\Omega(T,P)^{\ad}$) is \[\dim\overline{\{o_1\}}=N.\]

On the other hand, consider the following commutative diagram:
\[\xymatrix{
\Omega(T,P)^{\ad} \ar[d]^{\pi} \ar@{^{(}->}[r] & \Fl(G, \mu)_{L}^{\ad} \ar[d]^{\gamma} \\
\Omega(T,P)^{\Berk} \ar@{^{(}->}[r] & \Fl(G,\mu)_{L}^{\Berk}
}\] Then $\pi^{-1}(o)=\gamma^{-1}(o)$ and it is closed in $\Fl(G,\mu)_{L}^{\ad}$. Thus the previous closure $\overline{\{o_1\}}$ is indeed the closure inside the whole space $\Fl(G,\mu)_{L}^{\ad}$ and its dimension is $N$.

Since the Newton stratification factors through the Berkovich quotient, for each $b \in B(G,\mu^{-1})$ and each point $y \in \Fl(G,\mu)_{L}^{\Berk, b}$, its fiber $\gamma^{-1}(y)$ is closed and is contained in the same strata $\Fl(G,\mu)_{L}^{\ad, b}$. In particular, we have $o_1\in \Fl(G,\mu)_{L}^{\ad, b} \Leftrightarrow \overline{\{o_1\}} \subset \Fl(G,\mu)_{L}^{\ad, b}$.
Since $\Fl(G,\mu)_{L}^{\ad, b}$ is a locally spectral space,  its dimension is equal to the maximal length of specializing chains. 
Recall the dimension formula for the Newton strata (see Proposition \ref{prop Newton} (2))), for each $b \in B(G,\mu^{-1})$,
\[\dim\ \Fl(G,\mu)^{\ad,b}=\langle \mu-\nu(b), 2 \rho \rangle.\]  In particular, for any non basic $b$ (equivalent to $\nu(b)$ is not central), we have \[ \dim\ \Fl(G,\mu)^{\ad,b}<N=\langle \mu, 2 \rho\rangle.\] Then the dimension of the union of non basic strata is  strictly smaller than $N$. Therefore $o_1$ lies in the basic stratum. Then $o$ lies in the basic stratum. We are done.

\end{proof}

\begin{rem}
In the de Rham setting, i.e. the setting of \cite{RZ96} Chapter 1 or \cite{DOR10} Part 3, in \cite{hartl2013} Hartl 
made a conjecture (Conjecture 6.4 of \cite{hartl2013}) comparing the difference at level of points:
for any non-archimedean field $L|\breve{E}$ with a finitely generated valuation group, we have \[\Fl(G,\mu^{-1}, b_0)^{\mathrm{wadm}}(L)=\Fl(G,\mu^{-1}, b_0)^{\adm}(L).\]
This conjecture is a generalization of the classical results that $\Fl(G,\mu^{-1}, b_0)^{\mathrm{wadm}}$ and $\Fl(G,\mu^{-1}, b_0)^{\adm}$ have the same classical points (consequence of the theorem of Colmez-Fontaine). 
Hartl proved its analogue  for function fields in \cite{hartl2011}.

In our setting, there is a natural variant\footnote{We will explain how to translate our setting back to the de Rham setting in section 6.} of Hartl's conjecture: for any non-archimedean field $L|E$ with a finitely generated valuation group, we have \[\Fl(G,\mu)^{\Berk, \HN=b_0}(L)=\Fl(G,\mu)^{\Berk, b_0}(L).\]
Our result is compatible with this conjecture. In fact, we can also give another proof for Theorem \ref{main1} by Hartl's conjecture and Theorem \ref{weak}. The idea is in fact similar.

Take a finite extension $L|F$ to split $G$. For any apartment of $\B(G,L)$, apply the proposition \ref{explicit} to  explicit the Berkovich map, then any point $x$ inside the apartment will correspond to a generalized Gauss point $\theta(x)$. Its residue field $K=\mathcal{H}(\theta(x))$ is an affinoid extension of $L$. In particular, the valuation group is still finitely generated. By Theorem \ref{weak}, we have \[\theta(x)\in\Fl(G,\mu)^{\Berk, \HN=b_0}.\] If Hartl's conjecture holds, then $\Fl(G,\mu)^{\Berk, \HN=b_0}(K)=\Fl(G,\mu)^{\Berk, b_0}(K)$, and thus $\theta(x)$ lies in $\Fl(G,\mu)^{\Berk, b_0}$.  

\end{rem}

\section{Comparison of boundaries}
\label{boundary}
In this section, we extend the comparison of buildings and $p$-adic period domains to the boundaries. We keep our notations as before.

\subsection{Functorialities of Newton stratifications}\label{subsection functorial Newton}
 First, we point out that the Newton stratification is functorial.
 
 Let $M$ be an $F$-Levi subgroup of $G$. Suppose over $\bar{F}$, there is a minuscule cocharacter $\mu_{M}$ for $M_{\bar{F}}$ such that its composition with $M_{\bar{F}}\hookrightarrow G_{\bar{F}}$ is the minuscule cocharacter $\mu$. Then $P_{\mu_M}=M \bigcap P_{\mu}$ and this induces an embedding \[\Fl(M,\mu_M)\hookrightarrow \Fl(G,\mu)\] over $E$.
 
 We reformulate the above embedding by Bruhat decomposition. Let $P$ be an $F$-parabolic  subgroup of $G$ with associated Levi subgroup $M$. Choosing a maximal torus $T$ and a Borel subgroup $B$ over $\ov{F}$, and a dominant representative $\mu\in X_\ast(T)_+$ of the conjugacy class $\{\mu\}$.  Let $W$ be the absolute Weyl group of $G$, with the subgroups $W_P$ and $W_{P_\mu}$ corresponding to the parabolic subgroups $P_{\ov{F}}$ and $P_\mu$. Consider the set of minimal length representatives ${}^PW^{P_\mu}\subset W$ for the coset $W_P\setminus W/W_{P_\mu}$. Let $X_\ast(T)_{M,+}$ be set of $M$-dominant cocharacters of $T$, i.e. cocharacters of $T$ dominant for the induced Borel $B\cap M_{\ov{F}}$ of $M_{\ov{F}}$.  Then for each $w\in {}^PW^{P_\mu}$ we get an $M$-dominant cocharacter $\mu^w:=w\mu\in X_\ast(T)_{M,+}$. The $P_{\ov{F}}$-orbits on $\Fl(G,\mu)_{\ov{F}}$ define the Bruhat decomposition
 \[\Fl(G,\mu)_{\ov{F}}=\coprod_{w\in{}^PW^{P_\mu} } P_{\ov{F}}wP_\mu/P_\mu, \]
 and for each $w\in{}^PW^{P_\mu} $, the projection $P\ra M$ induces a morphism of algebraic varieties over $\ov{F}$
 \[ P_{\ov{F}}wP_\mu/P_\mu\lra \Fl(M,\mu^w),\]
 which is a fiberation in affine bundles of rank $\ell(w)$ (for example see \cite{DOR10} Lemma 6.3.6). In particular, the minimal (closed) stratum which corresponds to $w=\tr{id}$ is given by $\Fl(M, \mu_M)$, and we get the closed embedding $\Fl(M,\mu_M)\hookrightarrow \Fl(G,\mu)$ over $\ov{F}$, which descends to $E$.
 
 \begin{lem}\label{lem newton sets}
  Let $M$ be an $F$-Levi subgroup of $G$, and $\mu_M$ a cocharacter of $M$ over $\ov{F}$ with induced cocharacter $\mu$ of $G$ under the inclusion $M\subset G$. Then the induced map $B(M)\ra B(G)$ between Kottwitz sets restricts to a map
  \[B(M,\mu_M)\ra B(G,\mu).\]
 \end{lem}
 \begin{proof}
 If $G$ is quasi-split over $F$, then it is clear that we get an induced map $B(M,\mu_M)\ra B(G,\mu)$. Indeed, this follows from the definitions of the sets $B(M,\mu_M)$ and $B(G,\mu)$: under the induced maps $\pi_1(M)_\Gamma\ra \pi_1(G)_\Gamma$ and $\N(M)\ra \N(G)$, we have $\mu_M^\sharp\mapsto \mu^\sharp$ and $\mu_M^\diamond\mapsto \mu^\diamond$. Moreover, the partial orders on $\N(M)$ and $\N(G)$ are compatible, since both can be described in terms of relative coroots under the quasi-split assumption.
 
 Next, we reduce the general case to the quasi-split case. Assume that $G$ is non quasi-split, and let $H$ be its quasi-split inner form over $F$. Then $G$ corresponds to an element  \[\xi\in H^1(F, H_{ad})=\pi_1(H_{ad})_\Gamma=[\lan \Phi\ran ^\vee/\lan \Phi^\vee\ran ]_\Gamma.\] 
 Moreover, under the bijection \[H^1(F, H_{ad})\simeq B(H_{ad})_{basic},\] we get an induced $[b_G]\in B(H_{ad})_{basic}$ such that $G_{ad}=H_{ad, b_G}$. Now, the $F$-parabolic subgroups of $G_{ad}$ corresponds to the $F$-parabolic subgroups of $H_{ad}$ which admit a reduction of $b_G$, see \cite{cfs2021} Definition 2.5 and section 7. Let $M_{ad}$ be the induced Levi of $G_{ad}$ and $M^H_{ad}$ the corresponding Levi of $H_{ad}$. The isomorphism class of $M_{ad}$ defines an element \[\xi_M\in H^1(F, M^H_{ad}),\] which is given by the class of the reduction $b_M$ of $b_G$ to $M^H_{ad}$. Thus under the map $H^1(F, M^H_{ad})\ra H^1(F, H_{ad})$ we have $\xi_M\mapsto \xi$.  The inclusion $M^H_{ad}\subset H_{ad}$ induces a map
 \[B(M^H_{ad}, \mu^\sharp+_M\xi_M,\mu^\diamond_M)\ra B(H_{ad},\mu^\sharp+\xi,\mu^\diamond).\]
 On the other hand, we have natural bijections
 \[B(M,\mu_M)=B(M_{ad},\mu_M)=B(M^H_{ad}, \mu^\sharp_M+\xi_M,\mu^\diamond_M)\] and similarly
 \[B(G,\mu)=B(G_{ad},\mu)= B(H_{ad},\mu^\sharp+\xi,\mu^\diamond).\]
 Here we use the fact that the inner twisting induces bijections (cf. \cite{cfs2021} subsection 4.2)
 \[B(G)\cong B(H),\quad B(G,\epsilon, \delta)\cong B(H, \epsilon+\xi, \delta). \]
 Putting together, we get the induced map $B(M,\mu_M)\ra B(G,\mu)$.
 \end{proof}
 In the following, we will use the version of induced map $B(M,\mu_M^{-1})\ra B(G,\mu^{-1})$ by Lemma \ref{lem newton sets}.
 We have the following proposition:
 \begin{prop}
 	\label{natural}
 	The map $\Fl(M,\mu_M)^{\ad}\hookrightarrow \Fl(G,\mu)^{\ad}$ is compatible with Newton stratifications: if $b_M\mapsto b$ under the map $B(M,\mu_M^{-1})\ra B(G,\mu^{-1})$ by Lemma \ref{lem newton sets}, then the stratum $\Fl(M,\mu_M)^{\ad, b_M}$ is mapped into the stratum $\Fl(G,\mu)^{\ad, b}$ under the morphism $\Fl(M,\mu_M)^{\ad}\hookrightarrow \Fl(G,\mu)^{\ad}$. 
 \end{prop}
 
 \begin{proof}
 	The observation is that the Newton stratification is defined by composition of natural maps. More precisely, we have the following commutative diagram:
 	\[\xymatrix{
 		\Fl(M,\mu_{M})^{\ad,\Diamond} \ar[d] \ar[rr]^{\pi_{M,\mu_M}^{BB,-1}}& & \Gr_{M,\mu_M}^{B_{\dR}^{+}} \ar[r]^{BL_{1,M}} \ar[d] & \Bun_{M}  \ar[d] \\
 		\Fl(G,\mu)^{\ad,\Diamond}  \ar[rr]^{\pi_{G,\mu}^{BB, -1}} && \Gr_{G,\mu}^{B_{\dR}^{+}} \ar[r]^{BL_{1,G}} & \Bun_{G}.
 	}\]
 	Now this proposition is obvious.
 \end{proof}
 

 \subsection{Strongly regular elements in the Kottwitz set}

To compare boundaries of the Berkovich compactification with non basic Newton strata, we first introduce a notion of \emph{strongly regular elements}  in $B(G,\mu^{-1})$.

Let $P, M$ and $\mu_M$ be as in the last subsection. We have an embedding $\Fl(M,\mu_M)\hookrightarrow \Fl(G,\mu)$ over $E$.
By Lemma \ref{lem newton sets}, the group map $M\hookrightarrow G$ induces a map between Kottwitz sets $B(M,\mu_M^{-1})\rightarrow B(G,\mu^{-1}).$ 
\begin{definition}\label{def SR elements}
The basic element $b_M$ for $B(M,\mu_M^{-1})$ is mapped to an element $b$ of $B(G,\mu^{-1})$. We call this element $b$ \emph{strongly regular}. Letting $M$ vary, the subset of strongly regular elements is denoted by $\mathrm{SR}(G,\mu^{-1})$. 
\end{definition}
In practice to compute this subset, it is sufficient to consider standard $F$-Levi subgroups. 
Now we explain the name \emph{strongly regular}.
For such a $b$, let $M_b$ denote the centralizer of $\nu(b)$ in $G$. Then $M\subset M_b$ (and implicitly we also determine the compatible $F$ parabolic subgroup for $M_b$). The cocharacter $\mu_M$ further induces a cocharacter $\mu_1$ for $M_b$. Combining together, we have the following maps: \[B(M,\mu_{M}^{-1})\rightarrow B(M_b,\mu_1^{-1})\rightarrow B(G,\mu^{-1})\] Then it is easy to see $b_M$ is mapped to the basic element $b_{M_b}$ for $B(M_b,\mu_1^{-1})$. In particular, $b$ is a $G$-\emph{regular} element coming from $B(M_b)_{basic}$. See section 6 of \cite{kott1985}. Our requirement is much stronger, since we take into account of cocharacters $\mu^{-1}$ and $\mu_M^{-1}$. Indeed, this additional requirement has the following consequence.

We first make a comparison of strongly regular elements with \emph{Hodge-Newton decomposable} elements. Recall (cf. \cite{eva2024} Definition 7.1) for $b\in B(G)$ and $\delta\in X_\ast(A)^+_\Q=\N(G)$ with $\nu(b)\leq \delta$, we say $(b, \delta)$ is Hodge-Newton decomposable if there exists a proper standard Levi subgroup $M$ of the quasi-split inner form $H$ of $G$ which contains the centralizer of $\nu(b)$, such that \[\delta-\nu(b)\in \lan \Phi_{0,M}^\vee\ran_\Q.\] For $b\in B(G,\mu^{-1})$, it is called Hodge-Newton decomposable if $(b, \mu^{-1,\diamond})$ is Hodge-Newton decomposable in the above sense. If $G=H$ quasi-split,
Lemma 4.11 in \cite{cfs2021} gives some equivalent descriptions about Hodge-Newton decomposable elements $b$ inside $B(G,\mu^{-1})$. From the above discussion, a strongly regular element $b$ will come from the basic element of $B(M_b,\mu_1^{-1})$. Thus if $G$ quasi-split, it satisfies the condition of Lemma 4.11 (2) of \cite{cfs2021}, so it is a Hodge-Newton decomposable element. In general, we have
\begin{lem}\label{lem SR non basic}
If $b\in \mathrm{SR}(G,\mu^{-1})$ is the image of  $B(M,\mu_{M}^{-1})_{basic}\ra B(G,\mu^{-1})$ for a proper $F$-Levi subgroup $M$, then it is non basic.
\end{lem}
\begin{proof}
We pass to the adjoint quotient $G_{ad}$. This influences nothing.
Now $G$ is an adjoint group over $F$, then $G$ indeed has a product decomposition \[\prod_{i} G_i = G,\] where $G_i$ is the quasi-simple component of $G$. Then the cocharacter $\mu$ is also a product $\mu=\prod_{i} \mu_i$, where $\mu_i$ is a minuscule cocharacter for $G_i$. Then \[B(G,\mu^{-1})= \prod_{i} B(G_i,\mu_i^{-1}).\]
Therefore we only need to verify our claim in the quasi-simple case. 

Take a maximal split torus $S$  of $G$, consider the relative root system. Then the resulting Dynkin diagram is connected. For any parabolic subgroup $P$ together with its Levi $M$, performing the process of defining strongly regular element to get $\mu_M$,  if the basic element $b_M \in B(M,\mu_M^{-1})$ corresponds to the basic element of $B(G,\mu^{-1})$, then from the definition of  $B(M,\mu_M^{-1})$ we get 
\[\mu_M^{-1,\diamond}-\nu_M(b_M)=\sum c_i \alpha_i,\] where $c_i >0$ and $\alpha_i$ is a positive coroot of $M$. Passing to $B(G,\mu^{-1})$, if $b_M$ is mapped to the basic element, then $\nu(b_M)=0$ (since the center of $G$ is trivial), therefore (by our choice of $\mu_M$) \[\mu^{-1,\diamond}=\mu_M^{-1,\diamond}-\nu(b_M)=\sum_{i} c_i\alpha_i\] lies in the positive chamber. This is nonzero because the type determined by $\{\mu\}$ is non-trivial. On the other hand, because $M$ is proper, it can not contain  all positive coroots of $G$. Suppose it does not contain the positive coroot $\beta$ of $G$ and let $\widehat{\beta}$ denote the corresponding positive root. Then for each $\alpha_i$, the product $\langle \alpha_i, \widehat{\beta} \rangle$ is non-positive. Because the Dynkin diagram is connected, there exists at least one $i$ with such negative product. Then \[\langle \mu^{-1,\diamond}, \widehat{\beta} \rangle < 0,\] which is impossible.
\end{proof}

On the other hand, if $G$ quasi-split over $F$,
the requirement of being strongly regular is much stronger than being Hodge-Newton decomposable. 
\begin{example}
For the group $G=\GL_n$, we can view $B(G)$ as a subset of $\N(G)$ (since the Newton map is injective), thus describe an element $b$ through its Newton polygon $\nu(b)$. For non-basic $b \in B(G,\mu^{-1})$, the HN decomposable condition requires that $\nu(b)$ has a turning point lying in the Hodge polygon $\mu^{-1,\diamond}$. The strongly regular condition requires that \textbf{each} turning point of $\nu(b)$ also lies in the Hodge polygon of $\mu^{-1,\diamond}$. Obviously this requirement is much stronger and thus $\mathrm{SR}(G,\mu^{-1})$ contains very few elements. 
\begin{itemize}
	\item 
If $\mu=(1,...1,0..,0)$ with $d$ terms 1, then $\mathrm{SR}(G,\mu^{-1})$ contains only $d(n-d)+1$ elements. 
\item If $d=1$ or $d=n-1$ or $n \leq 4$, then $B(G,\mu^{-1})$ is fully Hodge-Newton decomposable and each element is also strongly regular. 
\item When $n=5$ and $d=2$, there is a single non-Hodge-Newton decomposable element. Other elements are also strongly regular. 
\item When $n \geq 6$, there will be many elements being Hodge-Newton decomposable but not strongly regular. To illustrate this more explicitly, we draw a picture showing an example with $n=7$ and $d=4$.  See the following Figure \ref{pic}.
\end{itemize}
\end{example}

\begin{figure}[h]

 \centering
\begin{tikzpicture}

\begin{axis}[ xlabel=x, 
ylabel=y, 
tick align=outside, 
 legend style={at={(0.5,-0.2)},anchor=north}%
 ]

 \addplot+ [sharp plot,mark=*,blue] plot coordinates {
 (0,0)  (1,1)   (3,3)   (6,3)   (7,3)
 };
  \addlegendentry{$\mu^{-1,\diamond}$}

 \addplot+ [sharp plot,mark=*,pink] plot coordinates {
 (0,0) (1,1) (3,2) (6,3) (7,3)
 };
 \addlegendentry{$\nu(b)$}

\end{axis}                         
\end{tikzpicture}                  
\caption{Example}
\label{pic}
\end{figure}

\subsection{Boundaries of Berkovich compactifications and non basic Newton strata}

Now we turn to the Berkovich compactification. The conjugacy class of the geometric minuscule cocharacter $\{\mu\}$ determines an $F$-type denoted as $t$. In Theorem \ref{thm BT vs p-adic period} we proved that \[\B_t(G,F)\subset \Fl(G,\mu)^{\Berk, b_0}.\] By Proposition \ref{levistrata}, the boundary of $\B_t(G,F)$ in
the Berkovich compactification $\ov{\B_{t}(G,F)}$ can be described by the Bruhat-Tits buildings of $\tau$-relevant  proper  Levi subgroups $M$, where $\tau$ is the $F$-rational type uniquely determined by $t$.
We can then write 
\[ \overline{\B_{t}(G,F)}=\bigcup_{M\,\tr{standard}}\bigcup_{g\in G(F)}g\Big(\B_{\tau}(M,F)\Big).\]
 On the other hand, the boundary of  $\Fl(G,\mu)^{\Berk, b_0}$ can be described by non basic Newton strata by Proposition \ref{prop Newton}. It turns out that the boundaries of  $\B_{t}(G,F)$ and $\Fl(G,\mu)^{\Berk, b_0}$ are matched neatly in the following way:
\begin{thm}\label{thm boundaries}
For each proper standard $F$-rational Levi group $M$ of $G$,
 the contribution of the boundary $\bigcup _{g\in G(F)} g\Big(\B_{\tau}(M,F)\Big)$ is contained in the Newton stratum corresponding an element $b_M \in \mathrm{SR}(G,\mu^{-1})$. This $b_M$ is determined by $M$, and each elements in $\mathrm{SR}(G,\mu^{-1})$ will appear in this way.  

\label{main2}
\end{thm}

\begin{proof}

Since the Newton stratification is invariant under the $G(F)$-action, for each conjugacy class of Levi subgroup, it is sufficient to prove the assertion  for $\B_{\tau}(M,F)$.

 Recall the proof of the Proposition \ref{levistrata}. To compute the contribution from $M$, we pick up an $F$-parabolic subgroup $P$ containing $M$ with Levi decomposition $P= N \rtimes M$. Take a finite extension $L|E$ to split $G$. Then $\{\mu\}$ determines an $L$-rational type $t_1$. We use $\Fl(G,\mu)_L$ to compute the Berkovich compactification $\overline{\B_{t}(G,F)}$. Choose a $L$-Borel subgroup $B_M$ of $M_L$ such that $N_L \rtimes B_M$ is a Borel subgroup for $G_{L}$. Also take a maximal split torus $T$ inside $B_M$, then $M_L$ is a standard Levi subgroup and $P_L$ is a standard parabolic subgroup. Take the dominant representative of $\{\mu\}$, again we still denote it as $\mu$. Then $P_{\mu}$ is compatible with $P_L$ and $M_L$, in other words, the natural inclusion induces an isomorphism \[M_L/(M_L \cap P_{\mu})\cong P_L/(P_L \cap P_{\mu}\] and $P_{\mu}$ is osculatory with $P_L$. Then $\B(M,F)$ contributes to the Berkovich compactification through the following map: \[\B(M,F)\xrightarrow{\theta} \Par_{t_1}(M_L)^{\Berk} \cong \Par_{t_1}(P_L)^{\Berk}\cong \Osc_{t_1}(P_L)^{\Berk}\hookrightarrow \Par_{t_1}(G_L)^{\Berk}.\] By Theorem \ref{main1}, the morphism  $\theta$ will send $\B(M,F)$ into the basic stratum of $\Par_{t_1}(M_L)^{\Berk}$. On the other hand, the cocharacter $\mu$ induces a minuscule cocharacter $\mu_M$ for $M_L$ and the natural map \[\Fl(M_L,\mu_M) \hookrightarrow \Fl(G_L,\mu)\] can be identified with the above map \[\Par_{t_1}(M_L) \hookrightarrow \Par_{t_1}(G_L),\] thus we can apply the functoriality of the Newton stratification (Proposition \ref{natural}), we see that the contribution $\B_{\tau}(M,F)$ will lies in the stratum corresponding to $b_M$, which is the image of $B(M,\mu_{M}^{-1})_{basic} \longrightarrow B(G,\mu^{-1})$.

Finally, the process of computing the contribution $\B_{\tau}(M,F)$ is the same as the process in the definition of strongly regular elements, so $b_M \in \mathrm{SR}(G,\mu^{-1})$ and each element of $\mathrm{SR}(G,\mu^{-1})$ will appear in such way.

\end{proof}

\begin{rem}

In practice, to compute $\mathrm{SR}(G,\mu^{-1})$ it is sufficient to consider those Levi $M$ appearing in the $\tau$-relevant parabolic subgroups. In \cite{rtw2010} subsection 3.3, there is a combinatorial description of $\tau$-relevant subgroups in terms of Dynkin diagrams. 
\end{rem}

A further study of $\mathrm{SR}(G,\mu^{-1})$ will show that the image of the Berkovich map is closed in the $p$-adic period domain. It is sufficient to check this in non-degenerate case. 
\begin{cor}
\label{closedcor}
\begin{enumerate}
	\item 
The image of Berkovich embedding $\B_t(G,F)$ is closed in the $p$-adic period domain $\Fl(G,\mu)^{\Berk, b_0}$.

\item Assume that $G$ is quasi-split over $F$. Then the image of Berkovich embedding $\B_t(G,F)$ is also closed in the open Harder-Narasimhan stratum $\Fl(G,\mu)^{\Berk, \HN=b_0}$.
\end{enumerate}
\end{cor}

\begin{proof}

The first statement follows from Lemma \ref{lem SR non basic}, since for any boundary stratum corresponding to a proper $M$ of $G$, its contribution $b \in \mathrm{SR}(G,\mu^{-1})$ can not be the basic element.

For the second assertion, we apply Theorem 1.3 of  \cite{eva2024}, which implies that if a non basic Newton stratum $\Fl(G,\mu)^{\Berk, b}$ intersects non trivially with the semistable locus $\Fl(G,\mu)^{\Berk, \HN=b_0}$, then $b$ is Hodge-Newton indecomposable. Under the assumption that $G$ is quasi-split, each non-basic element in $\mathrm{SR}(G,\mu^{-1})$ is Hodge-Newton decomposable. Therefore if a  contribution $\B_{\tau}(M,F)$ intersects with $\Fl(G,\mu)^{\Berk, \HN=b_0}$, it will lie in the basic locus. Then previous claim shows that $M$ has to be the whole group $G$. We are done.

\end{proof}

\begin{rem}\label{rem stronger closed}
One may prove Corollary \ref{closedcor} (2) without the quasi-split assumption by a different approach. More precisely, in Proposition \ref{weak} we have shown $\B_t(G, F)\subset \Fl(G,\mu)^{\Berk, \HN=b_0}$. One may prove a similar result for each boundary stratum $\B_\tau(M, F)$, by using the parabolic induction results for non basic Harder-Narasimhan strata, cf. \cite{DOR10} or \cite{shen2023} Theorem 3.9. As we will not need non basic Harder-Narasimhan strata in the following, we leave the details to the interested reader.
\end{rem}

\section{Retraction maps for $\GL_n$}
\label{retraction}

 In this section, we discuss  some special examples. We will mainly study the case $G=\GL_n$ with $n\geq 2$, though some results in fact hold for more general groups. Let $\mu$ be a minuscule cocharacter of $G$ of the form $(1^d, 0^{n-d})$ such that $(d, n)=1$. Since we will always work with Berkovich spaces in this section, we simply denote $\Fl(G,\mu)^{b_0}=\Fl(G,\mu)^{\Berk, b_0}$ and $\Fl(G,\mu)^{\HN=b_0}=\Fl(G,\mu)^{\Berk,\HN=b_0}$.
 Inspired by the works of \cite{pv1992} and \cite{vos2000}, we will construct a continuous retraction map \[r: \Fl(G,\mu)^{b_0}\lra \B(G,F).\] Here ``retraction'' means that the Berkovich embedding $\theta: \B(G,F)\hookrightarrow \Fl(G,\mu)^{b_0}$ by Theorem \ref{thm BT vs p-adic period} will be a section of $r$.  
 This map $r$ will generalize the Drinfeld map (cf. \cite{Dri74, Ber95}) in the case $d=1$ \[r:\Omega^{n}\lra \B(G,F).\]
 We will also discuss some analogy with tropical geometry and propose a new method to study the $p$-adic period domain $\Fl(G,\mu)^{b_0}$ and the semistable locus $\Fl(G,\mu)^{\HN=b_0}$.
 
 \subsection{$p$-adic period domain and the stable locus}
We first make a few remarks on comparison with the setting of \cite{pv1992} and \cite{vos2000}.
Both loc. cit. work with semisimple simply connected groups like $\SL_n$. Here we will use the adjoint group $\PGL_n$, as we want to use minuscule cocharacters etc. 
 Moreover, both \cite{pv1992} and \cite{vos2000} use some geometric invariant theory, which involves choosing a $G$-equivariant (very ample) line bundle $\mathcal{L}$ on the flag variety $G/P_\mu$ to talk about semistable and stable locus. Since $\SL_n\ra\PGL_n$ is an isogeny with a finite central kernel, we can always rescale $\cL$ by $\cL^{n}$ to get a $\PGL_n$-equivalent line bundle.  This will not influence the geometry of semistable or stable locus. In the rest of this section, we modify our notation by setting $G=\PGL_n$, the parabolic $P=P_\mu$ is defined by the induced minuscule cocharacter $\mu$ of $\PGL_n$.

Take a maximal split torus $T$ together with a Borel subgroup $B$, and consider the resulting root system and fundamental (positive) weights $\omega_i$. The maximal parabolic subgroups $P_i$ containing $B$ are bijection with the fundamental weights $\omega_i$. And any parabolic subgroup containing $B$ is in the form $\cap_i P_i$, thus corresponding to a non-empty subset $I$ of the set of fundamental weights. For such $P$,  one  can further consider a positive weight $\lambda=\sum_{i\in I}m_i \omega_i$ ($m_i$ is positive) and thus produce a highest weight representation $V=V_\lambda$ with a highest weight vector $e_{\lambda}$ (well defined up to units). 

In our case, the cocharacter $\mu$ corresponds to the fundamental weight $\lambda=\omega_d$.
Then we can identify the flag variety (which is the Grassmannian $\Gr(d, n)$)
 \[X:=\Fl(G,\mu)=G/P_\mu\] with the orbit $G(e_{\lambda})$ inside the projective space $\mathbb{P}(V)$ over $F$.  As in subsection \ref{subsection Newton and HN}, we have the associated $p$-adic period domain \[X^{b_0}\subset X^{\Berk}\] and semistable locus \[X^{ss}:=X^{\HN=b_0}\subset X^{\Berk}.\] Both are open subspaces of the Berkovich flag variety $X^{\Berk}$. Moreover, by Proposition \ref{prop HN} we have
\[X^{b_0}\subset X^{ss}.\] On the other hand, there is a variant of $X^{ss}$, the stable locus $X^{s}$ which classifies stable objects in the corresponding Harder-Narasimhan theory (for filtered vector spaces with $G$-structure or admissible modifications of $G$-bundles). By construction we have an open immersion
\[X^s\subset X^{ss}.\]

The closed embedding $X\hookrightarrow \mathbb{P}(V)$ further induces a very ample line bundle $\cL$ (pullback of $O(1)$ on $\mathbb{P}(V)$) on the flag variety $X$ and $\cL$ is also a $G$-equivariant line bundle. Use this line bundle $\cL$ and geometric invariant theory (for the $T$-action on $X$), we can study the semistable and stable locus respect to $T$ \[X(T,\cL)^{ss} \quad \tr{and} \quad X(T,\cL)^s.\] Both are open subschemes of $X$.
When $T$ varies, these describe the above semistable and stable locus:
\begin{prop}\label{prop GIT ss and s}
We have the following equalities of Berkovich analytic spaces over $F$
\[X^{ss}=\bigcap_T X(T,\cL)^{ss, \Berk}, \quad X^s=\bigcap_TX(T,\cL)^{s, \Berk},\]
where $T$ runs over all maximal split torus.
\end{prop}
\begin{proof}
This follows from \cite{pv1992} Proposition 2.6, Corollary 2.8.2, and the discussions in section 3 there. Note that in \cite{pv1992}, they work with rigid analytic spaces, and they introduced and studied $X(T,\cL)^{ss}$ and $X(T,\cL)^{s}$ first, then proceeded to show there are rigid analytic structures on the intersections over all $T$. Moreover, there are moduli interpretations for the intersections (the right hand side above) given by the semistable and stable objects (the left hand side above).  See also \cite{DOR10} Theorem 9.7.3 (whose proof in turn origins from \cite{totaro}; the later proves the conjecture in \cite{RZ96} 1.51, which in turn was motivated by \cite{pv1992}).
\end{proof}

To study $X(T,\cL)^{ss}$ and $X(T,\cL)^{s}$, following \cite{pv1992} we further introduce certain convex hulls. The action of $T$ on $H^0(X,\cL)$ produce the weight decomposition \[V^*=H^0(X,\cL)=\bigoplus_{\beta}H^0(X,\cL)_{\beta},\] where $\beta \in X^*(T)$.   For any geometric point $x \in X(k)$ ($k|F$ an algebraically closed field), evaluate $V^*$ at $x$, and let \[S_x \subset X^*(T)\otimes \mathbb{R}\] be the set of those $\beta$ with $V_{\beta}^*$ being non-vanish at $x$. Let \[Conv(S_x)\subset X^*(T)\otimes \mathbb{R}\] denote its convex hull. The Lemma 1.2 in \cite{pv1992} shows that $x \in X(T,\cL)^{ss}$ if and only if $0 \in Conv(S_x)$, and $x \in X(T,\cL)^{s}$ if and only if $0$ is an interior point of $Conv(S_x)$. Applying this lemma and Lemma 1.3 of \cite{pv1992}, we have
\begin{prop}\label{prop ss equals s}
Assume that the cocharacter $\mu=(1^d,0^{n-d})$ satisfies the condition $(d, n)=1$. Then
for any maximal split torus $T$, we have
 \[X(T,\cL)^s=X(T,\cL)^{ss}.\]
 In particular, under this assumption, by Proposition \ref{prop GIT ss and s} we get
 \[X^{ss}=X^s. \]
\end{prop}
\begin{proof}
	See \cite{pv1992} Theorem 1.1 and Corollary 2.4.
\end{proof}
The conclusion $X(T,\cL)^s=X(T,\cL)^{ss}$ is quite strong, which will be crucial for the following constructions.
From now on, we assume that the cocharacter $\mu=(1^d,0^{n-d})$ satisfies the condition $(d, n)=1$.
 As a corollary, we get a $G(F)$-equivariant inclusion of open subspaces of $X^\Berk$
 \[X^{b_0}\subset X^s. \]
 To construct a retraction map $r: X^{b_0}\ra \B(G,F)$ for the embedding $\theta: \B(G,F)\hookrightarrow X^{b_0}$, it suffices to construct a retraction map \[r: X^s\lra \B(G,F).\] By restriction to $X^{b_0}$, we get the desired map.

\subsection{The retraction map for a maximal torus}\label{subsec retration appartment}

In section 3 of \cite{pv1992}, van der Put and Voskuil implicitly constructed a map from $X(T,\cL)^{s}$ to the corresponding apartment $A_T$ inside $\B(G,F)$. Their construction is through a point-wise  description, thus they need to pick up a test field $K$.  We will review their construction, complete their arguments at some points, and adapt it to the setting of Berkovich spaces.

Let $K|F$ be a non-archimedean extension with $K$ algebraic closed. In \cite{pv1992} only $K=\mathbb{C}_p$ was considered, but it is  important to go beyond $\mathbb{C}_p$ when working with Berkovich spaces. A point in the Berkovich space (or the algebraic variety) may be realized as a $K$-point for different test fields $K$, but it will clear from the context that the construction is well defined.



Let $v: K^* \rightarrow \mathbb{R}$ denote the additive valuation of $K$ extending the discrete valuation on $F$ and satisfying $v(\pi)=1$ (here $\pi$ is a uniformizer of $F$). Using the same sign convention of Tits in \cite{tits1979}, we get the following map \[v_T: T(K)\cong \Hom (X^*(T),K^*)\longrightarrow \Hom(X^*(T),\mathbb{R}).\]
Pick up a special vertex $o$ in the corresponding apartment $A_T$ for $T$, and use $o$ (as the origin) to identify $A_T$ with $\Hom(X^*(T),\mathbb{R})$. The vertex $o$ also defines an integral model $G_o$ over $O_F$ for $G$ by the Bruhat-Tits theory, which is a reductive group scheme since $o$ is special and $G=\PGL_n$ is adjoint. For simplicity we may also use $G$ to denote this integral model. Previous discussions ($T$, $P$, $G/P$, ...) hold over $O_F$.  In particular, we have an open subscheme\[X(T,\cL)^s\subset X\] of the scheme $X=G/P$ over $O_F$, which is the (properly) stable locus for the action of $T$ on $X$ with respect to $\cL$. 
Then we get an analytic subspace 
\[\wh{X(T,\cL)^s}^\Berk_\eta\subset X(T,\cL)^{s,\Berk}_\eta,\ \]
where $\wh{X(T,\cL)^s}$ is the $p$-adic completion of $X(T,\cL)^s$,  $\wh{X(T,\cL)^s}^\Berk_\eta$ is the Berkovich analytic generic fiber of the formal scheme $\wh{X(T,\cL)^s}$, and $X(T,\cL)^{s,\Berk}_\eta$ is the Berkovich analytification of the generic fiber of the $O_F$-scheme $X(T,\cL)^{s}$.  In the following we simply denote $X(T,\cL)^{s,\Berk}=X(T,\cL)^{s,\Berk}_\eta$ as the $F$-analytic space. For the field $K$ as above, note that we have
\[ \wh{X(T,\cL)^s}^\Berk_\eta(K)=\wh{X(T,\cL)^s}(O_K)=X(T,\cL)^s(O_K).\]

If necessary, replace $\cL$ by  $\cL^k$ for a suitable positive integer $k$, so that there exists an \textbf{integral} basis $\{f_1,...,f_m\}$ for $H^0(X,\cL)^{T}$ over $O_F$ with the property \[X(T,\cL)^{s}=\bigcup X_{f_i},\] where $X_{f_i}$ is the non-vanishing locus of $f_i$ and we view everything over $O_F$.
This implies that $\wh{X(T,\cL)^s}^\Berk_\eta$ is a finite union of affinoids inside $X^\Berk$.
Let  $\widetilde{K}$ denote the residue field. Applying the geometric invariant theory, we can perform the geometric quotient \[Z=X(T,\cL)^{s}/T,\] which is a projective scheme over $O_F$, since $X(T,\cL)^{s}=X(T,\cL)^{ss}$ by Proposition \ref{prop ss equals s} and our assumption $(d, n)=1$. Each  $K$-point of $X(T,\cL)^s$ has a finite stabilizer inside $T(K)$. As in section 3.4 of \cite{pv1992}, we have:
\begin{enumerate}
	\item 
$X(T,\cL)^s(O_K)/T(O_K)=Z(O_K)=Z(K)=X(T,\cL)^s(K)/T(K)$.

\item  $T(K) \times X(T,\cL)^s(O_K)\twoheadrightarrow X(T,\cL)^{s}(K)$. 
\end{enumerate}
We note that the equality $Z(O_K)=Z(K)$ and the surjectivity of (2) hold since $Z$ is projective over $O_F$, and the other two equalities in (1) hold since $Z$ is a geometric quotient.
These properties determine uniquely a map (which comes essentially from Proposition \ref{prop ss equals s} and our assumption $(d, n)=1$)
\[r_{T,o}: X(T,\cL)^s(K) \longrightarrow \Hom(X^*(T),\R) \] with the property
\begin{enumerate}
	\item 

 $r_{T,o}(X(T,\cL)^s(O_K))=0$, and 

\item for any $t\in T(K)$ and $x\in X(T,\cL)^s(K)$, 
 \[r_{T,o}(t.x)=-v_{T}(t)+r_{T,o}(x).\]
\end{enumerate}
Through the identification with origin $o$, we can further view the target as the apartment $A_T$ inside the Bruhat-Tits building $\B(G,F)$. The resulting map is still denoted by \[r_{T,o}: X(T,\cL)^s(K) \longrightarrow A_T.\]  Letting $K$ vary, we get a $T$-equivariant map from the associated Berkovich $F$-analytic space of $X(T,\cL)^s$ to $A_T$ \[r_{T,o}: X(T,\cL)^{s, \Berk} \longrightarrow A_T,\]
which we  call  the apartment retraction map. Here we use the name ``retraction" because its composition with the Berkovich map (restrict to the apartment) is identity, cf. Theorem \ref{retract} (for the global version of retraction map, which is harder). By construction, we have
\[ r_{T,o}^{-1}(o)=\wh{X(T,\cL)^s}^\Berk_\eta.\]

 For later upgrading into Berkovich space, we will always view the set of points $X(T,\cL)^{s}(K)$ (and other varieties) as a subspace of the Berkovich flag variety $X^{\Berk}$ and equip them with the subspace topology. Now we will show that the resulting map $r_{T,o}$ is a continuous map of topological spaces.

To prove such continuous results, we need to introduce more notions. The weight decomposition holds as $O_F$-modules \[H^0(X,\cL)=\bigoplus_{\beta}H^0(X,\cL)_{\beta},\] and the previous integral basis $\{f_1,...,f_m\}$ is the basis for weight $0$ space. For other nonzero $\beta$, we also pick up an integral basis $\{f_{\beta,1},...,f_{\beta,\beta_m}\}$. 
For each $\beta$ and each point $x \in X(K)$, we define the follow notion \[||x||_{\beta}=\max\{|f(x)|_x\, |\, f \in H^0(X,\cL)_{\beta}\}=\max\{|f_{\beta,i}(x)|_i\}.\] Here we evaluate the line bundle $\cL$ at $x$ with a norm $|\ |_x$. It is a one dimensional space over $K$, thus the norm is well defined up to positive scalar. In practice we can first put a reasonable non-archimedean metric on $\cL$. The flag variety can be covered by open cells isomorphic to the affine space $\A^{N}$ and the line bundle $\cL$ is induced by $O(1)$ through the embedding into the projective space\footnote{ The situation is the  same as complex geometry, we can endow such metric on $\cL$ just like the case of equipping canonical metric for projective space in complex geometry.} $\mathbb{P}(V)$.
Since we only need to care about some ratios like \[\frac{||x||_{\beta}}{||x||_{\gamma}},\] which are canonically defined, thus independent of choice of such metric. Here we observe that under Berkovich topology, for any $\beta$, the map $|X^\Berk|\ra \R_+,\, x \mapsto ||x||_{\beta}$ is continuous.
With the help of this function, we have the following identifications: 

\[X(T,\cL)^s(O_K)=r_{T,o}^{-1}(o)(K)=\{x\in X(T,\cL)^{s}(K) \,|\, ||x||_{\beta}\leq ||x||_0, \ any \  \beta \}.\]
Now we can verify the continuous property. Let $\mathbb{F}_q$ be the residue field of $O_F$.
\begin{prop}\label{prop continuous retract torus}
 This apartment retraction map  $r_{T,o}$ is continuous.
 \end{prop}

 \begin{proof}
Take a field $K|F$ as above.
 Apply Lemmas 1.2 and 1.3 in \cite{pv1992} cited previously to study the set \[S_{\ov{y}}\subset X^\ast(T)\otimes\R\]  for any $\widetilde{K}$-point $\ov{y} \in X(\widetilde{K})$  over the residue field $\widetilde{K}|\F_q$. Here $T$ denotes the reduction to $\F_q$ of the previous integral $T$. Then we can rescale $\cL$ by $\cL^k$ for a sufficient large $k$, such that for any $\ov{y}$ and any root $\alpha$, some positive multiple $\alpha_{\ov{y}} \alpha  $ will appear in $S_{\ov{y}}$. See the discussion in section 3.4 of \cite{pv1992} for this fact.
 
 Let $T(K)$ act on $\Hom(X^*(T), \R)$ through the translation action $t \mapsto -v_{T}(t)$. Then its actions on both sides are homeomorphisms. Then it is  sufficient to show the continuity at any integral point $x \in X(T,\cL)^s(O_K)$. Then $r_{T,o}(x)=0\in \Hom(X^*(T),\R)\cong A_T$.
 
 Then for any $\beta$, we know that $||x||_{\beta}\leq ||x||_0$ and $||x||_0$ is obviously nonzero. Thus for any positive number $\varepsilon <1$, by the continuity of the function $\frac{||-||_\beta}{||-||_0}$, there exists a neighborhood $U_{\varepsilon}$  of $x$ in $X(T,\cL)^{s}(K)$ such that for any $x_0 \in U_{\varepsilon}$, $||x_0||_0$ is nonzero and for any $\beta$, we have \[\frac{||x_0||_{\beta}}{||x_0||_0} < 1+\varepsilon.\]
 Suppose $x_0=t.y$ with $y \in X(T,\cL)^s(O_K)$ and $t\in T(K)$. Then we have for any $\beta$, \[||x_0||_{\beta}=|\beta(t^{-1})|||y||_{\beta}.\] Due to the beginning discussion for $S_{\bar{y}}$ (here $\bar{y}$ is the reduction to the residue field), we know that for any root $\alpha$, there is a positive integer $\alpha_{\ov{y}}$ such that $\alpha_{\ov{y}} \alpha \in S_{\bar{y}}$. In particular, this implies \[||y||_{\alpha_{\ov{y}} \alpha}=||y||_0.\]
 Putting together, we get \[|\alpha(t)|^{\alpha_{\ov{y}}} > \frac{1}{1+\varepsilon }> 1- \varepsilon,\] thus \[|\alpha(t)|>1-\varepsilon.\] On the other hand, we can  do the same argument for the opposite root $-\alpha$, and we can the desired control \[\frac{1}{1-\varepsilon}> |\alpha(t)|>1-\varepsilon.\]
 Then let $\varepsilon$ be small enough. Combining with the fact that the root system span the  whole space $X^*(T)\otimes \R$, this will force $r_{T,o}(x_0)=-v_{T}(t)$ close to $0$ arbitrarily. 
 
 Notice that the set of $K=\mathbb{C}_p$ points is dense in the Berkovich space, through the following lemma, the apartment retraction map
 \[r_{T,o}: X(T,\cL)^{s, \Berk} \longrightarrow A_T\]
  is continuous.
 \end{proof} 
 
 \begin{lem}\label{lem top}
 	Let $f: A \longrightarrow B$ be a map between topological spaces. Suppose that there exists a dense subset $A_0$ of $A$ such that for any point $x \in A$, the restriction $A_0 \cup \{x\} \longrightarrow B$ is continuous. If $B$ is a regular Hausdorff space, then $f$ is a continuous map.
 \end{lem} 
 This is a routine exercise in topology and we omit the proof. The Bruhat-Tits building $\B(G,F)$  is locally compact Hausdorff, thus it is regular Hausdorff, so we can apply this lemma to upgrade  the pointwise description into the Berkovich setting.

 \begin{rem}
 \begin{enumerate}
 	\item 
 In the above proof some base change functoriality is used implicitly. Through the extension $K|F$, we can identify $\B(G,F)$ with a subspace of $\B(G,K)$. The original apartment $A_T$ of $\B(G,F)$ is identified with the corresponding apartment $\widetilde{A_T}$  for $\B(G,K)$. Thus the group $T(K)$ can not act on $\B(G,F)$, but it acts on $A_T$ naturally by the identification $A_T=\widetilde{A_T}$, which is exactly the translation action defined in the proof.
 \item In fact, one can work with some large enough field $K$, e.g. certain maximally complete field in the sense of \cite{Poon}, to avoid the argument using the density of classical points and Lemma \ref{lem top}.
\end{enumerate} 
 \end{rem}

 Before going on, we prove some further properties of the map $r_{T,o}$.

 \begin{prop}\label{prop properties retract torus}
We have the following properties for the apartment retraction map:
\begin{enumerate}
	\item 
For any $g \in G(F)$, let $T_1$ denote the maximal split torus $gTg^{-1}$ and let $\widetilde{o}=g(o)$, then the translation by $g$ on the flag variety can identify their stable locus \[g(X(T,\cL)^s)=X(T_1,\cL)^s.\] Moreover, for any $x \in X(T,\cL)^{s}(K)$, we have \[g(r_{T,o}(x))=r_{T_1,\widetilde{o}}(g(x)).\]

\item The apartment retraction map $r_{T,o}$ is independent of choices of the special vertex $o$, and thus we also denote it by $r_{T}$ or $r_A$.
\end{enumerate}
 \end{prop}

 \begin{proof}

For the first statement, we first verify the $G(F)$-equivariant property for integral points.

Through the identification $\widetilde{o}=g(o)$, the resulting integral model $G_{\widetilde{o}}$ can be identified with $G_{o}$, then the $g$ translation produce  the following commutative diagram

\[\xymatrix{
X(T,\cL)^s(K) \ar[r]^{\cong} & X(T_1,\cL)^s(K)\\
X(T,\cL)^s(O_K) \ar[u] \ar[r]^{\cong} & X(T_1,\cL)^s(O_K) \ar[u]}\]

So the equality holds for $x \in X(T,\cL)^s(O_K)$. Further for such $x$ and any $t \in T(K)$, we have 
\[g(r_{T,o}(t.x))=g.t.r_{T,o}(x)=(gtg^{-1}).g.(o)=(gtg^{-1}).r_{T_1,\widetilde{o}}(g(x))=r_{T_1,\widetilde{o}}(gtx),\] so we have verified the first statement.

Here we implicitly use the base  change functoriality again. We identify the apartment $A_T$ for $\B(G,F)$ with the corresponding apartment $\widetilde{A_T}$ inside $\B(G,K)$ and do similar identification for another apartment $A_{T_1}$. Then $T(K)$ (resp $gT(K)g^{-1}$) acts on $A_T$ (resp. $A_{T_1}$) in the natural way.

The second statement follows from the first one. For any other special vertex $o_1$ inside the  apartment $A_T$, there exist $t \in T(F)$ such that $t(o)=o_1$ and such $t$ conjugation will not change the maximal torus $T$. Now apply the first statement for any $x \in X(T,\cL)^s(K)$, we get 
\[t.(r_{T,o}(x))=r_{T,o_1}(t.x)=t.r_{T,o_1}(x),\] therefore $r_{T,o}=r_{T,o_1}$.

\end{proof}

\subsection{The retraction map for $\GL_n$}

 Now we construct the (global) retraction map to the whole Bruhat-Tits building. The stable locus $X^s$ is an open  analytic subspace of $X^{\Berk}$. By Proposition \ref{prop GIT ss and s}, it is the intersection \[X^{s}=\bigcap_{g \in G(F)} X(gTg^{-1},\cL)^{s, \Berk}, \] or equivalently, it corresponds to the locus that is stable respect to any $F$-maximal split torus $T$ action. In particular, for each apartment $A$ of the Bruhat-Tits building, we have a map \[r_A: X^s\lra A.\] The problem is to prove that these $r_A$ are compatible when $A$ varies. We will prove it in the following lemma\footnote{Voskuil wrote this lemma in \cite{vos2000} as Proposition 3.4. But there the proof is too vague and it contains some mistakes, for example,  the reduction will lie in the normalizer of the torus instead of the centralizer. Here we give a new rigorous proof.}, which will guarantee some basic finiteness results and the continuity of the retraction map.

 \begin{lem}[\textbf{Compatibility lemma}]\label{lemma compatibility}

Let $A_1$ and $A_2$ be  two apartments inside the Bruhat-Tits building $\B(G,F)$. Let $z$ denote an interior point of their intersection $A_1 \cap A_2$ with respect to $A_1 \cup A_2$. For any stable point $x \in X^s$,  we have
\[r_{A_1}(x)=z\quad \Longleftrightarrow \quad r_{A_2}(x)=z.\]

 \end{lem}
 \begin{proof}
 
 It suffices to show one side implication: $r_{A_1}(x)=z$ will imply $r_{A_2}(x)=z$. 
 We will prove this statement in two steps: first, we deduce it for $x$ coming from some $p$-adic field; second, we generalize it to any points.

 Recall that the Bruhat-Tits building $\B(G,F)$ has a $G(F)$-invariant metric, which is unique up to scalar. We make the normalization of the metric such that two closest special vertices has distance 1. This is possible because such neighborhoods $\{z_1,z_2\}$ are transitive under $G(F)$-action. Moreover, for any finite extension $F_1|F$, suppose the ramification index is $e$, then we require the metric on $\B(G,F_1)$ satisfying that two closest special vertices has distance $\frac{1}{e}$. Then according to the results of Landvogt in \cite{lan2000} about functoriality of the Bruhat-Tits building, the embedding $\B(G,F)\hookrightarrow \B(G,F_1)$ is an isometry after a suitable normalization of the metric on $\B(G,F_1)$. And our convention on the metric is compatible, thus no need to renormalize the metric again.
 
 Suppose $A_1$ corresponds to the maximal split torus $T_1$ and $A_2$ corresponds to $T_2$. 
Because $z$ is an interior point of $A_1 \cap A_2$, there exists a real number $\varepsilon > 0$ such that \[D(z,\varepsilon) \cap A_1=D(z,\varepsilon) \cap A_2,\] here $D(z,\varepsilon)$ is the open ball centered at $z$ with radius $\varepsilon$.

Now suppose the point $x$ comes from  $p$-adic fields, in other words, there exists a finite extension $F_1|F$ such that $x \in X(T_1,\cL)^s(F_1)$. Because our initial construction of $X(T_1,\cL)^{s}$ and its quotient $X(T_1,\cL)^s/T_1$ is purely algebraic over $O_F$, use their geometric properties, we can replace $F_1$ by a finite extension, such that \[x=t.x_0\] with $t\in T(F_1)$ and $x_0 \in X(T_1,\cL)^s(O_{F_1})$.  Then by the definition of the apartment retraction map, the point $z$ is a special vertex inside $\B(G,F_1)$. Here we use base change to identify $\B(G,F)$ with a subspace of $\B(G,F_1)$. By our choose of metric, this is an isometry. Thus inside $\B(G,F_1)$ we still have \[D(z,\varepsilon) \cap A_1=D(z,\varepsilon) \cap A_2.\]

Take a totally ramified extension $F_2|F_1$ with large enough  ramification index $e$ such that we have \[e> \frac{1}{\varepsilon}.\] By the base change functoriality again, we embed $\B(G,F_1)$ into $\B(G,F_2)$. We can define the apartment retraction map in the same way as over $F$. And they are compatible. Thus we can argue over $F_2$.

Now $z$ is a special vertex for the apartment $A_1$, so we can take it as the (new) origin and perform the apartment retraction $r_{T_1,z}$. Then for this new integral model of $G$ and thus $G/P$ etc, the point $x$ is an integral point for $X(T_1,\cL)^{s}$. We will also use $z$ as the origin for $A_2$ and consider the apartment retraction $r_{T_2,z}$. It is sufficient to show that $x$ is also an integral point for $X(T_2,\cL)^s$, which is equivalent to require the reduction $\overline{x}$  lies in the semistable locus \[X(\overline{T_2},\overline{\cL})^s\] over the residue field.

Let the distance between two closest special vertices on $\B(G,F_2)$ be $\delta$. By our assumption on $F_2$, we know that $\delta < \varepsilon$. Let $\overline{D}(z, \delta)$ be the closed ball centered at $z$ with radius $\delta$. We have the following observation:

\textbf{Claim}: the small neighbor $\overline{D}(z, \delta) \cap A_1$  will \textbf{determine} the  reduction of the maximal split torus $\overline{T_1}$ over the residue field.

Now we show this claim. Recall $G=\PGL_n$. Let $V$ be a $n$-dimensional $F_2$-vector space so that the building $\B(G, F_2)$ can be identified with equivalent class of norms on $V$ (or the dual vector space $V^*$). Because $z$ is a special vertex, we can pick up a representative norm inside its equivalence and suppose it corresponds to a lattice $LC$. The apartment $A_1$ together with $z$ can determine a decomposition \[LC= \bigoplus_{i}O_{F_2}e_i\] (pick up an integral adapted basis). For each point in the Bruhat-Tits building, we always rescale the norm so that $|e_n|=1$, then we can pick up a  representative inside the equivalent class. Then intersection of the closed ball $\overline{D}(z,\delta)$ with $A_1$ has $1+2(n-1)$ special vertices. Besides $z$ itself, there are $2(n-1)$ special vertices inside this intersection. Under our convention, we look at the lattices corresponding to them, then there are $n-1$ special vertices corresponding to lattices $L_i\subset LC$ with \[L_i=\bigoplus_{j \neq i}\lan e_j\ran\bigoplus\lan\pi_{O_{F_2}}e_i\ran,\] here we consider the integral module and $i$ runs over positive integer smaller  than $n$, and other $n-1$ special vertices corresponds to lattices $\widetilde{L}_i \supset LC$ with \[\widetilde{L}_i=\bigoplus_{j \neq i}\lan e_j\ran\bigoplus\lan\frac{1}{\pi_{O_{F_2}}}e_i\ran.\]

Let $\widetilde{F_2}$ denote the  residue field and $\overline{LC}$ denote the reduction of $LC$, for the first kind of special vertices. Putting together, we get the following $\widetilde{F_2}$-linear map \[\psi: \overline{LC} \longrightarrow \bigoplus_{i}LC/L_i.\] Its kernel is exactly the line generated by $\overline{e_n}$. For the second kind of special vertices, for each $i$, consider the following map \[\psi_i:\overline{LC} \longrightarrow LC/(\pi_{O_{F_2}}\widetilde{L}_i), \] its kernel is the line generated by $\overline{e_i}$. 

In summary, the $2(n-1)$ closest special vertices around $z$ can determine the basis $\{\overline{e_1},...,\overline{e_n}\}$ for $\overline{LC}$ over the residue field, thus determine the reduction of the maximal split torus $\overline{T_1}$. So our claim is true.

 Then since $\overline{D}(z, \delta) \cap A_2=\overline{D}(z, \delta) \cap A_1$, we get $\overline{T_1}=\overline{T_2}=\overline{T}$. In particular, this shows that over the residue field, they define the same semistable locus: \[X(\overline{T_1},\overline{\cL})^s=X(\overline{T_2},\overline{\cL})^s.\] Therefore the reduction $\overline{x}$ will lie in the semistable locus, then $x$ is an integral point for $X(T_2,\cL)^s$. We finish the first step. 

For a general point $x$, suppose $r_{A_2}(x)=z_2 \neq z$. The space  $A_1 \cup A_2$ is Hausdorff, thus there exists a neighborhood $U_2$ for $z_2$ and a neighborhood $U$ for $z$ such that $U_2 \cap U=\emptyset$. Shrink  $U$ to make  it being contained in $A_1 \cap A_2$. Because both retraction maps $r_{A_1}$ and $r_{A_2}$ are continuous, there exists a neighborhood $U_3$ for $x$ such that $r_{A_1}(U_3)\subset U$ and $r_{A_2}(U_3)\subset U_2$. Because $p$-adic points are dense, there exists a point $x_3 \in U_3$ that comes from a $p$-adic field. Then $r_{A_1}(x_3)$ is an interior point of $A_1 \cap A_2$, and the first step tells us that \[r_{A_2}(x_3)=r_{A_1}(x_3),\] which is a contradiction. Therefore $r_{A_2}(x)=r_{A_1}(x)$. We are done.

 \end{proof}
 
 \begin{rem}
 \begin{enumerate}
 	\item 
In the above proof, 
 the first step indeed works for any field with discrete valuation. If the valuation is non-discrete, we can not talk about the closest pair of special vertices inside the Bruhat-Tits building. They can be arbitrarily closed to each other. 
 
 \item During the proof of the first step, if we suppose $A_2=g(A_1)$ with $g$ also stabilizes $z$, then $g$ is integral and its reduction $\overline{g}$ lies in the normalizer of the maximal split torus $\overline{T}$. This result about image of such reduction is optimal. Any element in this normalizer has an integral lift $g$ inside the normalizer of the integral torus, then  such $g$ stables $z$ and $A_1$, thus $z$ is obviously the interior point of the intersection. On the other hand, this also shows the difficulty of proving compatibility for different apartment retraction maps. If $z$ lies in the boundary of $A_1 \cap A_2$, then  we lose control of the  reduction $\overline{g}$ and it may not stabilize the semistable locus over the residue field.
 

  \end{enumerate}
 \end{rem}
 
 For any point $z$ in the Bruhat-Tits building $\B(G,F)$, consider a subset (``fiber over $z$'') $Y_{z}$ of $X^\Berk$: \[Y_z=\bigcap_{z \in A}r_{A}^{-1}(z),\] here $A$ runs over all apartments containing $z$. 
 
 \begin{cor}\label{cor fin intersection}
 The intersection $\bigcap_{z \in A}r_{A}^{-1}(z)$ is a finite intersection.
\end{cor}
\begin{proof}
 Indeed, we first define a subset $Z^*$ (the star set of $\{z\}$) of simplices of the simplicial complex $\B(G,F)$, \[Z^*=\{\bigtriangleup | \bigtriangleup \cap \{z\} \neq \phi\}.\] Because $\B(G,F)$ is locally compact, this is a finite set. (In practice we can also similarly define the star set for any compact subset of $\B(G,F)$. Such a notion also appeared in other works on Bruhat-Tits buildings, like  \cite{ss1997, Schneider, dat2006}.) Any apartment $A$ containing $z$ will also contain some simplex from $Z^*$ and thus determine a subset \[S_A\subset Z^*.\] Then such subset $S_A$ has only finite possibilities. Any two apartment $A$ and $B$ corresponding to the same subset $S_1$ will share a common open subset (only relies on $S_1$) $U_1$ containing $z$, thus $U_1$ is contained in the interior of $A \cap B$. In particular, for any such subset $S_i$, pick up a representative apartment $A_i$, then by Lemma \ref{lemma compatibility}, the intersection  $\bigcap_{z \in A}r_{A}^{-1}(z)$  is just the finite intersection \[\bigcap_{i}r_{A_i}^{-1}(z).\]
\end{proof} 
 We have the following result:
 \begin{prop}\label{prop fiber decomp}
 \begin{enumerate}
 	\item For each $z\in \B(G,F)$,  $Y_{z}$ is a finite union of affinoid subspaces of $X^\Berk$.
 	\item 
$Y_{z} \subset X^{s}$.
 
 \item $X^{s}=\bigcup Y_{z}$.  
\end{enumerate}
\end{prop}
\begin{proof}
The first statement follows from Corollary \ref{cor fin intersection} that $Y_z=\bigcap_{z \in A}r_{A}^{-1}(z)$ is a finite intersection, and each $r_{A}^{-1}(z)$ is a finite union of affinoids of $X^\Berk$ by constructions in subsection \ref{subsec retration appartment}. 

The assertions (2) and (3) follow from \cite{vos2000} Corollary 3.10, Theorem 3.11 and Proposition 3.18.
\end{proof}
 
 By Lemma \ref{lemma compatibility}, for $z_1\neq z_2$ in an apartment $A_T$, we have $r_{A_T}^{-1}(z_1)\cap r_{A_T}^{-1}(z_2)=\emptyset$. Therefore in Proposition \ref{prop fiber decomp} (2) we actually get a disjoint union \[X^s=\coprod_{z\in \B(G,F)}Y_z.\] 
 Thanks to this decomposition, the retraction map exists obviously now. We just send points in $Y_{z}$ to the point $z \in \B(G,F)$ and get the retraction map \[r: X^{s} \longrightarrow \B(G,F).\] Then $Y_z$ is exactly the  fiber $r^{-1}(z)$.  Let $x\in X^{s}$ with $r(x)=z$, i.e. $x\in Y_z$, then for any apartment $A$ containing $z$, we have $r(x)=r_A(x)$.
Because of the $G(F)$-equivariant property of the apartment retraction map, this retraction map is also $G(F)$-equivariant. And it is  in fact also continuous.  
 \begin{thm}\label{thm retract continuous}

 The retraction map $r: X^{s} \longrightarrow \B(G,F)$ is continuous.

 \end{thm}
 
 \begin{proof}
 
 For any $x \in X^{s}$, suppose $r(x)=z \in \B(G,F)$, then $x \in Y_{z}$. 
 
 Consider the star set $Z^*$ of $\{z\}$ again, and suppose we have chosen the representatives of apartments $A_i$ corresponding to subsets $S_i$ of $Z^*$. Then for each $S_i$, there exists a set $U_i$ containing $z$ such that for any apartment $A$ containing $z$, if the corresponding subset is $S_i$, then $A$ contains $U_i$ and $U_i$ is open in $A$.
 
 For any neighborhood $U$ of $z$, shrink $U$ if necessary so that $U \cap A_i$ is contained in $U_i$ for any $i$. Now for any $i$, consider the apartment retraction map \[r_{A_i}: X(T_i,\cL)^{s,\Berk}\longrightarrow A_i.\] Since this map is continuous by Proposition \ref{prop continuous retract torus} and $X^{s} \subset X(T_i,\cL)^{s,\Berk}$ with $r_{A_i}(x)=z$, there exists a neighborhood for $x$ inside $X^{s}$ such that $r_{A_i}$ will map this neighborhood into $U \cap A_i$. 
 
 Let $\widetilde{U}$ denote the intersection of these neighborhoods of $x$. For any point $y \in \widetilde{U}$, suppose $r(y)=w$. Then there exists an apartment $A$ containing $w$ and $z$. To see this, just notice that $z$ (resp. $w$) is contained in a simplex $\Delta_1$ (resp. $\Delta_2$),  and there exists an apartment passing through $\Delta_1$ and $\Delta_2$.
 
 Suppose this apartment $A$ corresponds to the subset $S_1$ of $Z^*$. Then $U_1$ is contained in the interior of $A \cap A_1$. Because $r_{A_1}(\widetilde{U}) \subset U \cap A_1 \subset U_1$, the  point $r_{A_1}(y)$ is an interior point of $A \cap A_1$, therefore by Lemma \ref{lemma compatibility} we have \[r_{A}(y)=r_{A_1}(y).\] On the other hand, we have $w=r(y)=r_A(y)$. Thus $w=r_{A_1}(y)\in U\cap A_1\subset U$. Therefore, we have shown that $r(\widetilde{U})\subset U$.
 This theorem holds. 
 
 \end{proof}

\begin{rem}
In \cite{pv1992}, a formal integral model over $O_F$ of $X^s$ was constructed, based on a gluing procedure which is related to simplicial structure of $\B(G,F)$. We will not need this integral model, as in general it  says nothing about $X^{b_0}$. 
\end{rem}
  
 Next we justify the name ``\emph{retraction}''.  Recall that by Theorem \ref{thm BT vs p-adic period}  we have the Berkovich map (which is an embedding in this case) \[\theta: \B(G,F)\hookrightarrow X^{b_0}\] and the inclusion $X^{b_0}\subset X^s$. We still denote the restriction of $r$ to $X^{b_0}$ by $r: X^{b_0}\ra \B(G,F)$. 
 \begin{thm}
 \label{retract}
 The composition of the Berkovich map $\theta: \B(G,F)\hookrightarrow X^{b_0}$ and the retraction map $r: X^{b_0}\ra \B(G,F)$ is an identity for the Bruhat-Tits building $\B(G,F)$.

 \end{thm}

 \begin{proof}
 
 For any point $z \in \B(G,F)$, suppose $x=\theta(z) \in X^{b_0}$. We need to show that \[x \in Y_{z}=r^{-1}(z).\] 
 
 Take an apartment $A$ containing $z$ and pick up a special vertex $o$ for $A$. 
 Let the test field $K$ be large enough so that $v(K^*)=\R$ and $\theta(o)\in X^{b_0}(K)$.
 Using this special vertex $o$ as the origin, we get an integral model for $G$ etc and we have the continuous map $r_{A}: X(T,\cL)^{s,\Berk}\ra A$ by Proposition \ref{prop continuous retract torus}. We first show that \[r_A(\theta(o))=o.\] 
 Then, we only need to show that $\theta(o)\in X(T,\cL)^s(O_K)$. Through Proposition \ref{explicit} computing the Berkovich map, the point $\theta(o)$ is the standard Gauss point for the open Bruhat cell $\Omega(T,P)$. And its reduction $\overline{\theta(o)}$ corresponds to the $\widetilde{K}$-point of $\overline{\Omega(T,P)}$ that is an injection of the polynomial ring. So it is the generic point of the open Bruhat cell over  $\F_q$, then it  is also the generic point of the flag variety $\overline{X}$ over $\F_q$.
   
On the other hand, recall that the stable locus $X(T,\cL)^s$ is a union of non-vanishing locus $X_{f_i}$. Therefore $\theta(o)$ lies in the integral stable locus if and only if $\overline{\theta(o)}$ lies in the stable locus over the residue field. But over such a field, the stable locus is cutting off a family of proper Zariski closed subset. In particular, the generic point of the whole flag variety will certainly lie in this stable locus. Therefore $r_A(\theta(o))=o$.
 
 Note  that both maps $\theta|_A: A\rightarrow X(T,\cL)^s(K)$ and $r_A: X(T,\cL)^s(K)\rightarrow A$ are $T(K)$-equivariant (although $T(K)$- can not act on the whole Bruhat-Tits building). Because $v(K^*)=\R$, $T(K)$ acts on $A$ transitively, thus the first map is well defined, and each point in the apartment will be sent to a $K$-point. Combining with $r_A(\theta(o))=o$, the $T(K)$-equivariant property forces the map $r_A\circ\theta$ to be an identity. In particular, we have $r_A(x)=z$.  
 Since $A$ is an arbitrary apartment containing $z$, we get $r(x)=z$.

 \end{proof}

 \begin{rem}
 	There is a significant difference between the Berkovich map $\theta$ and the retraction map $r$. By \cite{rtw2010}, we indeed can construct the Berkovich map $\theta$  for any non-archimedean field $k$ satisfying the functoriality assumption. In particular, if $G$ is split, then any non-archimedean field $k$ works. But for the retraction map $r$, we need $k$ to be locally compact. This fact guarantees that the Bruhat-Tits building is locally compact, which plays an essential role in the construction of $r$ as above.
 	
 	Indeed, the basic example of Drinfeld space (see the next subsection) already shows this difference. We can always use the interpretation of (semi)norms to define a map from the Drinfeld space to equivalent classes of norms on the dual vector space $V^*$. But for general $k$, the later space is \textbf{not} the Bruhat-Tits building of $\PGL_n$. In terms of norms, the Bruhat-Tits building corresponds to those norms on $V^*$ with an adapted basis (thus inside an apartment), but for general $k$, there exist norms on $V^*$ without such a basis.
 	
 \end{rem}
 

\subsection{The case $d=1$}

Our retraction map is a generalization of  the Drinfeld retraction map \[r^{Dr}: \Omega^n\lra \B(G,F)\] for Drinfeld space \cite{Dri74} in the case $d=1$, where one can give a very simple definition by restriction of norms. In this case, we have $X=\mathbb{P}^{n-1}$ and \[\Omega^n=X^{b_0}=X^s=X^{ss},\]
which is the complement of the union of all $F$-rational hyperplanes in $\mathbb{P}^{n-1}$.
 The embedding $\theta: \B(G,F)\hookrightarrow \Omega^n$ was previous constructed by Berkovich in \cite{Ber90} (see also \cite{Ber95})

 \begin{prop}\label{prop retract Drinfeld}
 For $d=1$,  the retraction map $r: \Omega^n\lra \B(G,F)$ coincides with the Drinfeld retraction $r^{Dr}$ in terms of norms.
 \end{prop} 
 \begin{proof}

  Let $V$ denote the $n$-dimension $F$-vector space with identification $G=\PGL(V)$ and further identify the flag variety $X=\Fl(G,\mu)$ with the projective space $\mathbb{P}(V)$. Let $V^*$ denote the dual space. The Drinfeld half-space $\Omega^n$ is the semistable locus of $X^{\Berk}$ cutting out off those lines lying in an $F$-rational hyperplane. And the Drinfeld retraction $r^{Dr}$ is a $G(F)$-equivariant continuous map to equivalent classes  of norms on $V^*$. Here $G(F)$ acts on $V^*$ through the dual action and then acts on norms on it. We further identify the Bruhat-Tits building $\B(G, F)$ with the space of equivalent classes of norms on $V^*$.
  
   Let $x \in \Omega^n$ and take an apartment $A$ containing the point $r(x)$ and $r^{Dr}(x)$. Choose a special vertex $o$ inside this apartment and take it as the origin. Suppose the apartment $A$ corresponds to the maximal split torus $T$. 
  We first show the compatibility over $o$, i.e. $r^{Dr}(r^{-1}(o))=o$.
  
  First, we can also identify $\B(G,F)$ with the space of equivalent classes of norms on $V$. Then the origin $o$ together with the apartment $A$ determine an $O_{F}$-lattice (unique up to scalar) $LC$ equipped with an integral decomposition \[LC=\bigoplus_{i}\lan e_i\ran.\] And through the identification with norms on $V^*$, the dual basis $\{e_1^*,..,e_n^*\}$ for $V^*$ is an integral adapted basis for the apartment $A$ and $o$ is the standard Gauss norm, i.e. each vector $e_i^*$ is norm one. The basis $\{e_1,...,e_n\}$ for $V$ gives a (homogeneous) coordinate for $\mathbb{P}(V)$ and  use $o$ to upgrade everything into integral setting. For any point $y \in r^{-1}(o)$, there exists a non-archimedean algebraic closed field $K$ containing $F$ with $y$ being a $K$-point. Because $y \in r^{-1}(o)$, we can pick up a representative coordinate \[y=[y_1:...:y_n]\] with each $y_i \in O_K$ and the reduction $\overline{y}$ lies in the Drinfeld half-space over the residue field $\F_q$. Then each $\overline{y_i}$ is nonzero so we have $|y_i|=1$. Indeed, each point inside $X(T,\cL)^{s}(O_K)$ will already have such norm one result. For each nonzero vector $v=\sum a_{i}e_i^*$ with $a_i \in F$, its norm respect to the Drinfeld retract is \[r^{Dr}(y)(v)=|\sum_i a_i y_i|.\] Then it is sufficient to show that this is the Gauss norm with respect to the basis $\{e_i^*\}$. In other words, we need to check \[|\sum_{i}a_i y_i|=\max\{|a_i|\}.\]
  
  For simplicity, suppose $a_1$,...,$a_k$ reach the maximum norm $\max\{|a_i|\}$, we only need to show $|\sum a_j y_j|=\max\{|a_j|\}$, here $j$ runs from $1$ to $k$. Multiply a suitable power of $\pi$ (the fixed uniformizer of ${O_F}$) so that  each $a_j$ lies in $O_{F}^*$, then consider the reduction: \[\overline{\sum_{j}a_j y_j}=\sum_j (\overline{a_j}) \overline{y_j}.\] Since the reduction $y$ lies in the Drinfeld half-space, it can not lie in any $\F_q$-rational hyperplane, this reduction of the sum is nonzero. Thus $\sum_j a_j y_j$ has norm one, so $r^{Dr}(y)=o$.

 Now we return to the point $x$. Enlarge the test field $K$ if necessary so that $x$ is also a $K$-point. 
 
 We proceed as before, and suppose $x$ is $(x_1,...,x_{n-1})=[x_1:...:x_{n-1}:1]$, here for simplicity we use the inhomogeneous coordinate. This is possible because each $x_i$ is nonzero indeed. Now we  make the convention that $e_{n}^*$ is always norm 1. Then we can pick up a unique representative norm for each point in the  apartment $A$. Consider the Drinfeld retract $r^{Dr}(x)$, it sends the vector $\sum a_i e_i^*$ to $|\sum_{i<n} a_{i}x_i+1|$. Note that our conventions are compatible, the vector $e_n^*$ is surely norm 1 under this norm. Because $\{e_i^{*}\}$ is an adapted basis, so we know that $r^{Dr}(x)$ is determined by the following $n-1$-tuple: \[(r^{Dr}(x)(e_1),...,r^{Dr}(x)(e_{n-1}))=(|x_1|,...,|x_{n-1}|).\]
 
 On the other hand, the  retraction map for $x$ can also be  computed by the apartment retraction map $r_{A,o}$. Now suppose $x=t.y_{0}$ with $y_0 \in X(T,\cL)^s(O_K)$, and  suppose $t=(t_1,...,t_{n-1},1)$ and $y_0=[y_1:..y_{n-1}:1]$. By previous observation, each $y_i$ is norm 1 element. Then we have \[|x_i|=|t_i y_i|=|t_i|.\] Therefore the apartment retraction map $r_{A,o}$ sends $x$ to the norm  on $V^*$  corresponding to the $n-1$-tuple $(|t_1|,...,|t_{n-1}|)$, which means exactly $r^{Dr}(x)=r(x)$. We are done.

 \end{proof}
\begin{rem}
In the Drinfeld case, Werner constructed an extension of the retraction map $r: \Omega^n\ra \B(G, F)$ to a continuous map $\ov{r}: \mathbb{P}^{n-1,\Berk}\ra \ov{\B_t(G, F)}$, for which the embedding $\ov{\B_t(G, F)}\hookrightarrow \mathbb{P}^{n-1,\Berk}$ is a section, cf. \cite{werner2004} section 6.  In the setting of Theorem \ref{retract}, it may be possible to extend the retraction map $r$ to a continuous map $\ov{r}: X^\Berk\ra \ov{\B_t(G, F)}$ for which the the embedding $\ov{\B_t(G, F)}\hookrightarrow X^{\Berk}$ is a section. Since we do not need it in the following, we will not do this here.
\end{rem}
 
\subsection{$p$-adic period domains and tropical geometry}

We conclude this section by making some analogy with tropical geometry, which offers a possible new aspect on the $p$-adic period domains $X^{b_0}$.
Intuitively, it is better to think  about the retraction map $r$ in a reverse way. Here we sketch how to imagine its fiber geometrically. In spirit this is close to the method in \cite{vos2000}.
 
For a special vertex $z \in \B(G,F)$ (otherwise we  need  to use base change to make it become special vertex), we consider its fiber \[r^{-1}(z)=Y_{z}\subset X^s\] 
under the retraction map $r: X^s\ra \B(G,F)$. 
Such $z$ will give  us an integral model $G_z$ for $G$ and enable us to upgrade everything into integral level. In particular, we can consider the reduction map (the flag scheme is proper). For any test field $K$, through the reduction  we get a map \[X(K)=\Fl(G,\mu)(K)=\Fl(G_z,\mu_z)(O_K)\longrightarrow \overline{\Fl(G,\mu)}(\widetilde{K}),\] then $r^{-1}(z)(K)$ is in fact the lift of semistable (= stable) locus over the residue fields. This fiber is closed in the whole semistable locus, it is a finite union of affinoids with the same dimension as the flag variety.

Such fiber-wise thoughts produce some interesting observations. For any apartment $A$, let $T$ denote the resulting maximal split torus over $F$. Let us  restrict to the slice $r^{-1}(A)\subset X^s$ (inverse image over $A$):\[r^{-1}(A)\longrightarrow A\cong \R^{n-1}.\] An amazing thing is that both  $r^{-1}(A)$ and $A$ have an extra $T(K)$ action. The latter one is through the translation, and may think it acts through the identification of $A$ with the base changed apartment inside $\B(G,K)$. But the huge group $T(K)$ can also acts on $r^{-1}(A)$. Of course $T(K)$ can not act on the semistable locus or the Bruhat-Tits building. But inside this slice $r^{-1}(A)$, the $T(K)$ action can even identify different fiber (keeping analytic structure). Moreover, if the field $K$ satisfies $v(K^*)=\R$, then its action on the apartment is transitive. Therefore each fiber is isomorphic to each other. The above map looks like a ``fiberation". The Bruhat-Tits building is a union of apartments, so the whole retraction map also looks like a ``fiberation".

In a certain sense, this picture has some analogy with tropical geometry.
Consider the $n$-dimensional torus $\mathbb{G}_{m}^n$ (as a  Berkovich space), through the valuation, we get a continuous map \[r: \mathbb{G}_{m}^n \longrightarrow \R^{n},\] which has a skeleton section \[s:\R^n \longrightarrow \mathbb{G}_{m}^n\] sending a vector to the corresponding generalized Gauss point.  More generally, for other  Berkovich space $Y$, we may consider its map to $\mathbb{G}_m^n$, then composite with $r$ to get a subset $Im(Y)$ inside $\R^n$ and this set is called tropical image. These ideas (tropical image, skeleton section etc) are widely used in tropical geometry and other related fields. For example  see \cite{ducros2012}, \cite{tropical2017}, \cite{rabin2012}, \cite{werner2016} and so on. Although our setting is a little different from the usual cases in tropical geometry, there people usually use one formal (or integral) model, while the Bruhat-Tits building will provide ``variation" of integral models for $G$.

In the case of Drinfeld spaces, two pictures are perfectly compatible. For example, take a maximal split torus $T$, so we can take homogenous coordinate for the  projective space, then this induces an open embedding \[\Omega ^{n} \hookrightarrow \mathbb{G}_m^{n-1}\] through the inhomogenous coordinate (we have $n$-choices). Then the tropical map from $\Omega ^{n}$ to $\R^{n-1}$ is exactly the apartment retraction map $r_{A}$. And the skeleton section map is the Berkovich map. Moreover, in \cite{werner2011}, Werner introduced a tropical viewpoint on the Bruhat-Tits building and its compactification. Combining these ideas together, it may be possible to use tropical geometry (combined with the Berkovich map and the retraction map) to study the semistable locus $X^{ss}=X^s$.

Finally, we consider $X^{b_0}$. An important  question in $p$-adic Hodge theory is to understand the difference $X^{ss}\setminus X^{b_0}$. This object is mysterious, for example, it does not have any classical points. The previous discussion about the semistable locus is harder for $X^{b_0}$. For  example, the intersection \[X^{b_0}\cap r^{-1}(z)\] ($z$ is a special vertex) is mysterious, it is not a finite union of affinoids. The intersection broke the affinoid property, then the construction of formal model in \cite{pv1992} does not work for $X^{b_0}$. Moreover, since  $X^{b_0}$ shares the same classical points with $X^{ss}$, its reduction through $z$ to mod $p$ points will be all semistable points. 
On the other hand, we can look at the subspace $r^{-1}(\Delta) \subset X^{ss}$ over any maximal dimensional simplex. It is also a finite union of  affinoids. Even better, it is $F$-affinoids. The simplicial decomposition \[X^{ss}=\bigcup_{\triangle} r^{-1}(\triangle)\] induces a similar decomposition for $X^{b_0}$: we have \[X^{b_0}=\bigcup_{\triangle} \Big(X^{b_0} \cap r^{-1}(\triangle)\Big)=\bigcup_{g \in G(F)} g( X^{b_0}\cap \triangle_0),\] here we take a standard simplex $\triangle_0$. Therefore it reduces the problem of studying $X^{b_0}$ to study a slice \[X^{b_0}\cap r^{-1}(\triangle_0).\] The fiber $r^{-1}(\triangle_0)$ is a finite union of $F$-affinoids and is closely related with period domain (semistable locus) over the residue field $\F_q$, which may  be simpler.

 \section{Translations to the de Rham side}\label{section de Rham}
 In this section, we briefly explain how to translate the previous constructions and results to the de Rham setting as in the beginning of the introduction, i.e. the setting as in \cite{RZ96} Chapter one. Roughly speaking, we
 get similar results for the Bruhat-Tits building $\B(G_b, F)$ and the 
 admissible locus $\Fl(G,\mu^{-1},b)^\adm$ with $b=b_0\in B(G,\mu^{-1})$ basic. Here $G_b$ is the reductive group over $F$ defined by the $\sigma$-centralizers of $b$, cf. \cite{RR} 1.11 and \cite{kott1997} 3.3 (where it is denoted by $J_b$) and \cite{FS} III.4.1. As we assume $b$ is basic, $G_b$ is an inner form of $G$ over $F$.
 
 Starting with a basic local Shimura datum $(G, \{\mu^{-1}\}, b)$, we get a tower of rigid analytic spaces \[\Big(\M(G,\mu^{-1},b)_K\Big)_K\] over $\breve{E}$, where $K$ runs over the set of open compact subgroups of $G(F)$ and $E$ is the reflex field $E=E(G, \{\mu^{-1}\})$ as before. These rigid analytic spaces $\M(G,\mu^{-1},b)_K$ are called the associated local Shimura varieties.
 Let \[\Fl(G,\mu^{-1},b)^\adm\subset \Fl(G,\mu^{-1})^\ad_{\breve{E}}\] be the admissible locus as defined in \cite{sw2020}, which is the open Newton stratum for the Newton stratification on $\Fl(G,\mu^{-1})^\ad_{\breve{E}}$ as in \cite{cs2017, cfs2021}. Then there are \'etale morphisms (the de Rham period morphisms) of adic spaces over $\breve{E}$ \[\pi_{\dR, K}:   \M(G,\mu^{-1},b)_K\lra \Fl(G,\mu^{-1},b)^\adm,\]  which are surjective. Consider the local Shimura variety at infinite level \[\M(G,\mu^{-1},b)_\infty=\varprojlim_K\M(G,\mu^{-1},b)_K^\Diamond,\] defined as the inverse limit of the associated diamonds $\M(G,\mu^{-1},b)_K^\Diamond$ of $\M(G,\mu^{-1},b)_K$. The diamonds $\M(G,\mu^{-1},b)_K^\Diamond$ and $\M(G,\mu^{-1},b)_\infty$ are moduli spaces of certains $p$-adic $G$-shtukas, cf. \cite{sw2020} Lectures 23 and 24.
 There is the Hodge-Tate period morphism (cf. \cite{sw13, cs2017, cfs2021})
 \[\pi_{\mathrm{HT}}:  \M(G,\mu^{-1},b)_\infty\lra \Fl(G, \mu)^{b, \Diamond}_{\breve{E}},\]
 where \[\Fl(G, \mu)^{b}\subset \Fl(G, \mu)^{\ad}\] is the open Newton stratum of the adic space $\Fl(G, \mu)^{\ad}$ over $E$ studied in subsection \ref{subsection Newton and HN}. 
 
 Consider also the basic local Shimura datum $(G_b, \{\mu\}, b^{-1})$, which is the dual local local Shimura datum of $(G, \{\mu^{-1}\}, b)$ in the sense of \cite{sw2020} Corollary 23.3.2. Then one has a natural $G(F)\times G_b(F)$-equivariant isomorphism \[\M(G,\mu^{-1},b)_\infty\cong \M(G_b,\mu,b^{-1})_\infty \] of locally spatial diamonds over $\mathrm{Spd}\,\breve{E} =(\Spa\,\breve{E})^\Diamond$, cf. \cite{sw2020} Corollary 23.3.2.
 Moreover, we have natural isomorphisms of adic spaces over $\breve{E}$ (cf. \cite{shen2023} section 4):
 \[\Fl(G,\mu^{-1},b)^\adm\cong \Fl(G_b, \mu^{-1})^{b^{-1}}_{\breve{E}}, \quad  \Fl(G, \mu)^{b}_{\breve{E}}\cong \Fl(G_b,\mu,b^{-1})^{\adm}. \]
 These fit into a twin towers diagram 
 using the de Rham and Hodge-Tate period morphisms that allow us to collapse each tower on its base (cf. \cite{Fal02, FGL08, sw13, sw2020, cfs2021, shen2023})
 $$
 \begin{tikzcd}[row sep=large,column sep=large]
 	\M (G,\mu^{-1},b)_\infty \ar[d,"\pi_{dR}", two heads]\ar[r,"\sim"]  \ar[rd,"\pi_{HT}" description]
 	\ar[d, dash, dotted , bend right=40, start anchor={[xshift=-12mm]}, end anchor={[xshift=-12mm]}, start anchor={[yshift=3mm]}, end anchor={[yshift=-3mm]},"G(F)"']
 	& \M (G_b,\mu,b^{-1})_\infty \ar[d,"\pi_{dR}", two heads] 
 	\ar[ld,"\pi_{HT}" description]    \ar[d, dash, dotted, bend left=40, start anchor={[xshift=12mm]}, end anchor={[xshift=12mm]}, start anchor={[yshift=3mm]}, end anchor={[yshift=-3mm]},"G_b(F)"]
 	\\
 	\Fl(G,\mu^{-1},b)^{\adm,\Diamond} & \Fl(G_b,\mu,b^{-1})^{\adm,\Diamond}   
 \end{tikzcd}
 $$
 
 Now we can translate our previous constructions and results as follows:
 \begin{itemize}
 	\item (Theorem \ref{thm BT vs p-adic period}) The natural Berkovich map similarly constructed as before by R\'emy-Thuillier-Werner \[\theta: \B(G_b, F)\lra \Fl(G_b, \mu^{-1})_{\breve{E}}^\Berk=\Fl(G, \mu^{-1})_{\breve{E}}^\Berk\] factors through the admissible locus $\Fl(G,\mu^{-1},b)^{\Berk,\adm}$, i.e. \[\theta(\B(G_b, F))\subset \Fl(G,\mu^{-1},b)^{\Berk,\adm}.\]
 	In terms of moduli, this means that from a point of the building $\B(G_b, F)$, one can construct a $p$-adic $G$-shtuka (up to isogeny) in the sense of \cite{sw2020}. Note that here the base field $\breve{E}$ still satisfies the functoriality assumption of  \cite{rtw2010} 1.3.4.
 	\item (Theorem \ref{thm boundaries}) Let $\dot{b}$ be a representative of $b$ and $s$ an integer such that $s\nu_{\dot{b}}$ factors through $\G_m$ and  $E_s=E\cdot F_s$ with $F_s|F$ the unramified extension of degree $s$ (see \cite{RZ96} pages 8-9).  Then the space $\Fl(G,\mu^{-1},b)^{\adm}$ is defined over $E_s$ (this is an exercise by the construction of the admissible locus) and we have $G_{b,F_s}\cong G_{F_s}$ (cf. \cite{RZ96} Corollary 1.14). Consider the Berkovich map \[\theta: \B(G_b, F)\lra \Fl(G_b, \mu^{-1})_{E_s}^\Berk=\Fl(G, \mu^{-1})_{E_s}^\Berk.\] By the above, we have \[\theta(\B(G_b, F))\subset \Fl(G,\mu^{-1},b)^{\Berk,\adm}.\] Consider the closure \[\ov{\theta(\B(G_b, F))}\]  of $\theta(\B(G_b, F))$ inside $\Fl(G, \mu^{-1})_{E_s}^\Berk$. Then there exists a natural description of the boundary strata of $\ov{\theta(\B(G_b, F))}$ in terms of  $\B(M, F)$ with $M$ the $F$-rational proper Levi subgroups of $G_b$. Moreover, each boundary stratum $\B_\tau(M, F)$  is contained in a uniquely determined non basic Newton stratum.
 	\item (Section \ref{retraction}) Consider the basic local Shimura datum 
 	\[(D^\times, \{\mu\}, b^{-1}),\] where $D$ is a division algebra over $F$ of invariant $\frac{d}{n}$ with $(d, n)=1$, and $\mu=(1^{d}, 0^{n-d})$. Then $G_{b^{-1}}=\GL_n$ and the dual local Shimura datum is $(\GL_n, \{\mu^{-1}\}, b)$. Moreover, 
 	 \[\Fl(D^\times, \mu, b^{-1})^\adm=\Fl(\GL_n, \mu)^{b}\] is the $p$-adic period domain studied in section \ref{retraction} after base change to $\breve{E}$.  
 	We have \[\theta(\B(\GL_n, F))\subset \Fl(D^\times, \mu, b^{-1})^{\Berk,\adm}.\] Moreover, we have a continuous retraction map
 	\[r:  \Fl(D^\times, \mu, b^{-1})^{\Berk,\adm}\lra \B(\GL_n, F),\]
 	generalizing the case of Drinfeld for $d=1$.  Here, one can first construct the associated formal Rapoport-Zink space \[\widehat{\M},\] which is a moduli space of $p$-divisible groups with certain $O_D$-action together with a rigidification (cf. \cite{RZ96}), then one can construct the 
 	local Shimura varieties $\M_K$ as the \'etale coverings of the rigid analytic generic fiber of $\widehat{\M}$. In this case, the associated admissible locus $\Fl(D^\times, \mu, b^{-1})^{\adm}$, together with its \'etale coverings $\M_K$ and the isomorphism between the associated twin towers, were studied  in \cite{Fal10}. In particular, one sees that the whole picture is quite similar to the Drinfeld case, although the retraction map $r:  \Fl(D^\times, \mu, b^{-1})^{\Berk,\adm}\lra \B(\GL_n, F)$ is more complicated.
 \end{itemize}

Finally, we discuss some open problems. Back to the Hodge-Tate setting as in section \ref{newtonstrata}. Recall that we have the closed subspace $\B_t(G,F)\subset \Fl(G,\mu)^{\Berk,b}$ defined as the image of the building $\B(G,F)$ under the Berkovich map. Without loss of generality, assume that $\{\mu\}$ is non-degenerate, so that $\B(G, F)\cong \B_t(G,F)$, which we simply identify. Let \[\B(G, F)^\ad\subset \Fl(G, \mu)^{\ad,b}\] be the inverse image of $\B(G, F)= \B_t(G,F)$ under the quotient map\footnote{Note here we can replace the $p$-adic period domain by the whole flag variety.} $\Fl(G, \mu)^{\ad,b}\ra \Fl(G,\mu)^{\Berk,b}$. This is a generalizing closed subset of $|\Fl(G, \mu)^{\ad,b}|$, thus it defines a closed locally spatial diamond \[\B(G, F)^\Diamond\subset \Fl(G, \mu)^{\ad,b, \Diamond}\] such that \[|\B(G, F)^\Diamond|=\B(G, F)^\ad.\]  The diamond $\B(G, F)^\Diamond$ has a well defined \'etale site, cf. \cite{Sch17}. This arises a natural question to study its \'etale cohomology. More precisely, let $\ell\neq p$ be another prime, and $\Lambda\in \{\F_\ell, \Z_\ell, \Q_\ell\}$, the \'etale cohomology with compact support (\cite{Sch17})
\[H^\ast_{c}(\B(G, F)^\Diamond_{\ov{E}}, \Lambda) \]
form natural $\Lambda$-representations of $\tr{Gal}(\ov{E}/E)\times G(F)$. The question is  what are these representations? The classical compact support cohomology of the topological building $H^\ast_c(\B(G,F), \C)$ had been computed explicitly by Borel-Serre in \cite{BS76}. There only $H^d_c(\B(G,F), \C)\neq 0$ with $d=\dim\,\B(G,F)$, in which case it is the Steinberg representation of $G(F)$. Here the situation is much more complicated, as the \'etale cohomology involves the Galois cohomology of large (transcendent degree) non-archimedean fields $\mathcal{H}(\theta(x))$ for $x\in \B(G,F)$.

Similarly, one may study the derived category of $G(F)$-equivariant \'etale sheaves (\cite{Sch17}) \[D_{\tr{et}}(\B(G, F)^\Diamond, \Lambda).\] In the topological setting, there  is the work of Schneider-Stuhler \cite{ss1997}, in which one transforms smooth representations of $G(F)$ to sheaves (or coefficient systems) on the building $\B(G, F)$ and can deduce several interesting results on representations. In \cite{Schneider}, Schneider further studied Verdier duality of constructible sheaves on $\B(G, F)$. Inspired by the recent work of Fargues-Scholze \cite{FS}, it seems reasonable to study certain class of \'etale sheaves and duality on $\B(G, F)^\Diamond$, and deduce representation theoretic consequences.  We leave these considerations to future works.


\begin{thebibliography}{10}
	
		\bibitem{Ber90}V. G. Berkovich,
		\emph{Spectral theory and analytic geometry over non-archimedean fields},
	Mathematical Surveys and Monographs 3,
	Amer. Math. Soc. 1990.
	
	\bibitem{Ber93}V. G. Berkovich, \emph{\'Etale cohomology for non-Archimedean analytic spaces}, Publications Math. de l'IH\'ES, tome 78, 5-161, 1993.

	
	\bibitem{Ber95} V. G. Berkovich,
		\emph{The automorphism group of the Drinfeld half-plane},
	C.R. Acad. Sci. Paris S\'er. I Math. 321,
1127-1132, 1995.
	
	\bibitem{ext2022}
	Christopher Birkbeck, Tony Feng, David Hansen, Serin Hong, Qirui Li, Anthony
	Wang, and Lynnelle Ye.
	\newblock Extensions of vector bundles on the Fargues-Fontaine curve.
	\newblock {\em Journal of the Institute of Mathematics of Jussieu},
	21(2):487--532, 2022.
	
	\bibitem{BS76}A. Borel, J.P.-Serre,
		\emph{Cohomologie d'immeubles et de groupes S-arithm\'etiques},
	Topology 15,  211-232,  1976.
	

	\bibitem{bt1972}
	Fran{\c{c}}ois Bruhat and Jacques Tits.
	\newblock Groupes r{\'e}ductifs sur un corps local: I. donn{\'e}es radicielles
	valu{\'e}es.
	\newblock {\em Publications Math{\'e}matiques de l'IH{\'E}S}, 41:5--251, 1972.
	
	\bibitem{bt1984}
	Fran{\c{c}}ois Bruhat and Jacques Tits.
	\newblock Groupes r{\'e}ductifs sur un corps local: II. sch{\'e}mas en groupes.
	existence d'une donn{\'e}e radicielle valu{\'e}e.
	\newblock {\em Publications Math{\'e}matiques de l'IH{\'E}S}, 60:5--184, 1984.
	
	\bibitem{bt84a}
	Fran{\c{c}}ois Bruhat and Jacques Tits.
	\newblock Sch{\'e}mas en groupes et immeubles des groupes classiques sur un
	corps local.
	\newblock {\em Bulletin de la Soci{\'e}t{\'e} Math{\'e}matique de France},
	112:259--301, 1984.
	
	\bibitem{bt87}
	Fran{\c{c}}ois Bruhat and Jacques Tits.
	\newblock Sch{\'e}mas en groupes et immeubles des groupes classiques sur un
	corps local. ii: groupes unitaires.
	\newblock {\em Bulletin de la Soci{\'e}t{\'e} Math{\'e}matique de France},
	115:141--195, 1987.
	
	\bibitem{cs2017}
	Ana Caraiani and Peter Scholze.
	\newblock On the generic part of the cohomology of compact unitary Shimura varieties.
	\newblock {\em Annals of Mathematics}, 186(3):649--766, 2017.
	
	\bibitem{cfs2021}
	Miaofen Chen, Laurent Fargues, and Xu~Shen.
	\newblock On the structure of some $ p $-adic period domains.
	\newblock {\em Cambridge Journal of Mathematics}, 9(1):213--267, 2021.
	
	
	\bibitem{dat2006}
	Jean~Fran\c{c}ois Dat.
	\newblock Espaces sym\'etriques de {Drinfeld} et correspondance de {Langlands}
	locale.
	\newblock {\em Annales scientifiques de l'\'Ecole Normale Sup\'erieure}, 4e
	s{\'e}rie, 39(1):1--74, 2006.
	
	\bibitem{DOR10} Jean Fran\c{c}ois Dat, Sascha Orlik,  and Michael Rapoport, 
		\emph{Period domains over finite and $p$-adic fields},
		Cambridge Tracts in Mathematics 183, 2010.
		
	\bibitem{Dri74} V.G. Drinfeld,
			\emph{Elliptic modules},
		Math. USSR Sbornik 2,  561-592, 1974.
		
		\bibitem{Dri76} V.G. Drinfeld,
			\emph{Coverings of $p$-adic symmetric regions},
			Funct. Anal. Appl. 10, 107-115,  1976.
	
	
	\bibitem{ducros2012}
	Antoine Ducros.
	\newblock Espaces de Berkovich, polytopes, squelettes et th{\'e}orie des
	mod{\`e}les.
	\newblock {\em Confluentes Mathematici}, 4(4), 2012.
	
	\bibitem{Fal02} Gerd Faltings,
		\emph{A relation between two moduli spaces studied by V.G. Drinfeld},
	Contemp. Math. 30, 115-129, 2002.
	
	
	
	
	\bibitem{Fal10} Gerd Faltings,
	\emph{Coverings of $p$-adic period domians},
		J. reine angew. Math. 643, 111-139,  2010.
	
	\bibitem{FGL08}L. Fargues, A. Genestier, V. Lafforgue,
		\emph{L'isomorphisme entre les tours de Lubin-Tate et de Drinfeld}, Birkh\"auser,  Basel 2008.
		
		\bibitem{FS}L. Fargues, P. Scholze, \emph{Geometrization of the local Langlands correspondence}, arXiv: 2102.13459.
	
\bibitem{Hub96} Roland Huber,
		\emph{\'Etale cohomology of rigid analytic varieties and adic spaces}, Aspects of Mathematics, Springer 1996.
	
	
	
	\bibitem{tropical2017}
	Walter Gubler, Joseph Rabinoff, and Annette Werner.
	\newblock Tropical skeletons.
	\newblock In {\em Annales de l'Institut Fourier}, volume~67, pages 1905--1961,
	2017.
	
	\bibitem{hartl2011}
	Urs Hartl.
	\newblock Period spaces for Hodge structures in equal characteristic.
	\newblock {\em Annals of Mathematics}, pages 1241--1358, 2011.
	
	\bibitem{hartl2013}
	Urs Hartl.
	\newblock On a conjecture of Rapoport and Zink.
	\newblock {\em Inventiones Mathematicae}, 193(3):627--696, 2013.
	
	\bibitem{KL}K. S. Kedlaya, R. Liu, \emph{Relative $p$-adic Hodge theory: Foundations}, Ast\'erisque 371, Soc. Math. France, 2015.
	
	\bibitem{kott1985}
	Robert~E Kottwitz.
	\newblock Isocrystals with additional structure.
	\newblock {\em Compositio Mathematica}, 56(2):201--220, 1985.
	
	\bibitem{kott1997}
	Robert~E Kottwitz.
	\newblock Isocrystals with additional structure. II.
	\newblock {\em Compositio Mathematica}, 109(3):255--339, 1997.
	
	\bibitem{lan2000}
	Erasmus Landvogt.
	\newblock Some functional properties of the Bruhat-Tits building.
	\newblock {\em Journal f\"ur die reine und angewandte Mathematik},
	2000(518):213--241, 2000.
	
	\bibitem{pv1992}
	H. Voskuil, M. van der Put,
	\newblock Symmetric spaces associated to split algebraic groups over a local
	field.
	\newblock {\em Journal f\"ur die reine und angewandte Mathematik}, 433:69--100,
	1992.
	
	\bibitem{Poon} Bjorn Poonen, \emph{Maximally complete fields}, L'Enseign. Math. 39, 87-106, 1993.
	
	\bibitem{rabin2012}
	Joseph Rabinoff.
	\newblock Tropical analytic geometry, Newton polygons, and tropical intersections.
	\newblock {\em Advances in Mathematics}, 229(6):3192--3255, 2012.
	
	\bibitem{Rap} Michael Rapoport, \emph{Non-Archimedean period domains}, In Proceedings of the International Congress of Mathematicians,
	Vol. 1, 2 (Z\"urich, 1994), pages 423-434.
	
	\bibitem{rap1997}
	Michael Rapoport.
	\newblock Period domains over finite and local fields.
	\newblock In {\em Proceedings of Symposia in Pure Mathematics}, volume~62,
	pages 361--382. American Mathematical Society, 1997.
	
	\bibitem{RR}M. Rapoport, M. Richartz, \emph{On the classification and specialization of $F$-isocrystals with additional structure}, Compositio Math. 103 (1996), no. 2, 153-181.
	
	\bibitem{RZ96} Michael Rapoport, Th. Zink,
		\emph{Period spaces for $p$-divisible groups},
		Annals of Mathematics Studies 141,
		Princetion University Press, 1996.
	
	\bibitem{rtw2010}
	Bertrand R{\'e}my, Amaury Thuillier, and Annette Werner.
	\newblock Bruhat-Tits theory from Berkovich's point of view. I. realizations
	and compactifications of buildings.
	\newblock In {\em Annales Scientifiques de l'Ecole Normale sup{\'e}rieure},
	volume~43, pages 461--554, 2010.
	
	\bibitem{rtw2015}Bertrand R{\'e}my, Amaury Thuillier, and Annette Werner,  \emph{Bruhat-Tits buildings and analytic geometry}, in  Antoine Ducros, Charles Favre, Johannes Nicaise Editors,   ``Berkovich spaces and applications'', Lecture Notes in Mathematics 2119, 
	pages 141-202.
	
	\bibitem{ss91} Peter Schneider and Ulrich Stuhler,
		\emph{The cohomology of $p$-adic symmetric spaces},
		Inventiones Mathematicae
		105,
		47-122,  1991.
	
	
	\bibitem{ss1997}
	Peter Schneider and Ulrich Stuhler.
	\newblock Representation theory and sheaves on the Bruhat-Tits building.
	\newblock {\em Publications Math{\'e}matiques de l'IH{\'E}S}, 85:97--191, 1997.
	
	\bibitem{Schneider} Peter Schneider, \emph{Verdier duality on the building}, J. reine angew. Math. 494, 205-218, 1998.
	
	\bibitem{sch2012}
	Peter Scholze.
	\newblock Perfectoid spaces.
	\newblock {\em Publications Math{\'e}matiques de l'IH{\'E}S}, 116(1):245--313,
	2012.
	
	\bibitem{Sch17}Peter Scholze,
		\emph{\'Etale cohomology of diamonds},
	arXiv: 1709.07343.
	
	
	\bibitem{sw13}Peter Scholze and Jared Weinstein,
		\emph{Moduli of $p$-divisible groups},
		Cambridge Journal of Mathematics
		1, no. 2, 145-237, 2013.
	
	
	\bibitem{sw2020}
	Peter Scholze and Jared Weinstein.
	\newblock {\em Berkeley Lectures on $p$-adic Geometry: (AMS-207)}.
	\newblock Princeton University Press, 2020.
	
	\bibitem{shen2023}
	Xu~Shen.
	\newblock Harder-Narasimhan strata and $p$-adic period domains.
	\newblock {\em Transactions of the American Mathematical Society},
	376(05):3319--3376, 2023.
	
	\bibitem{tits1979}
	Jacques Tits.
	\newblock Reductive groups over local fields.
	\newblock In {\em Automorphic forms, representations and L-functions (Proc.
		Sympos. Pure Math., Oregon State Univ., Corvallis, Ore., 1977), Part},
	volume~1, pages 29--69, 1979.
	
	\bibitem{totaro}
	Burt Totaro.
	\newblock Tensor products in $ p $-adic Hodge theory.
	\newblock {\em Duke Math. J.}, 85(1):79--104, 1996.
	
	\bibitem{eva2024}
	Eva Viehmann.
	\newblock On Newton strata in the $B_{\dR}^+$-Grassmannian.
	\newblock {\em Duke Mathematical Journal}, 173(1):177--225, 2024.
	
	
	\bibitem{vos2000}
	Harm Voskuil.
	\newblock Non-archimedean flag domains and semistability I, preprint, arXiv:math/0011274.
	
	\bibitem{werner2004}Annette Werner, \emph{Compactification of the Bruhat-Tits building of $\PGL$ by seminorms}, Math. Z. 248, 511-526, 2004.
	
	\bibitem{werner2011}
	Annette Werner,
	\emph{A tropical view on Bruhat-Tits buildings and their compactifications}, Cent. Eur. J. Math. 9(2):390--402, 2011.
	
	\bibitem{werner2016}
	Annette Werner.
	\newblock Analytification and tropicalization over non-archimedean fields.
	\newblock In M. Baker, S. Payne (eds)  {\em Nonarchimedean and Tropical Geometry},  pages 145--171.
	Springer, 2016.
	
	
	\bibitem{Yu} Jiu-Kang Yu, \emph{Bruhat-Tits theory and Buildings}, in ``Ottawa Lectures on Admissible Representations of Reductive $p$-adic Groups'', C. Cunningham, M. Nevins, editors. Fields Institute Monographs,  2009.
	
	
\end{thebibliography}

\end{document}